\documentclass[a4paper,12pt,fleqn]{smfart}

\usepackage{ucs}
\usepackage[utf8x]{inputenc}

\usepackage[francais,english]{babel} \AutoSpaceBeforeFDP
\usepackage[T1]{fontenc}
\usepackage{amstext,amssymb,amsthm}
\usepackage{varioref}

\usepackage{layout}
\usepackage{smfthm}
\usepackage{graphicx}
\usepackage{subfig}
\usepackage[usenames]{color}
\usepackage{epsfig}
\usepackage[all]{xy}
\usepackage{enumerate}
\usepackage{version}

\usepackage{MnSymbol}
\usepackage{wasysym}

\usepackage{url}
\usepackage{geometry}
\definecolor{darkred}{rgb}{0.9,0.,.2}
\definecolor{darkblue}{rgb}{0.,0.,.6}
\definecolor{darkgreen}{rgb}{0.,.6,0.1}
\usepackage[colorlinks=true,urlcolor=darkblue,linkcolor=darkred,citecolor=darkgreen]{hyperref}

\newcommand{\C}{\mathcal{C}}
\newcommand{\R}{\mathbb{R}}
\newcommand{\N}{\mathbb{N}}
\newcommand{\Z}{\mathbb{Z}}
\newcommand{\Q}{\mathbb{Q}}

\renewcommand{\ss}{\mathrm{SL_{n+1}(\mathbb{R})}}
\newcommand{\SO}{\mathrm{SO_{n,1}(\mathbb{R})}}
\renewcommand{\H}{\mathbb{H}}

\newcommand{\Mc}{\M^{\dagger}}

\newcommand{\U}{\mathcal{U}}
\newcommand{\F}{\mathcal{F}}

\newcommand{\M}{\mathcal{M}}
\newcommand{\Nn}{\mathcal{N}}
\newcommand{\V}{\mathcal{V}}

\renewcommand{\O}{\Omega}
\newcommand{\g}{\gamma}
\newcommand{\G}{\Gamma}

\renewcommand{\S}{\mathbb{S}^n}

\newcommand{\PP}{\mathbb{P}}
\newcommand{\E}{\mathcal{E}}

\newcommand{\Aut}{\textrm{Aut}}

\newcommand{\Stab}{\textrm{Stab}}

\newcommand{\dev}{\textrm{Dev}}
\newcommand{\hol}{\textrm{Hol}}

\newcommand{\Quo}{\Omega/\!\raisebox{-.80ex}{\ensuremath{ \,\Gamma}}}

\theoremstyle{plain}
\newtheorem{fait}{Fait}
\newtheorem*{thm}{Théorème}

\theoremstyle{definition}

\theoremstyle{remark}
\newtheorem{nota}{Notations}

\title[Exemples de variétés projectives strictement convexes de volume fini]{Exemples de vari\'et\'es projectives strictement convexes de volume fini en dimension quelconque}
\author{\href{mailto:ludovic.marquis@umpa.ens-lyon.fr}{Ludovic Marquis}}
\date{} 
\urladdr{\href{http://www.umpa.ens-lyon.fr/~lmarquis/}{www.umpa.ens-lyon.fr/~lmarquis/}}

\begin{document}

\newcommand\point{\textbullet}
\newcommand*\points[1]{%
 \ifcase\value{#1}\or
   $\cdot$ \or $\udotdot$ \or $\therefore$ \or
  $\diamonddots$  \or $\fivedots$ \or $\medstarofdavid$ \or $\octagon$
 \fi}
\renewcommand\theenumi{%
 \points{enumi}}
\renewcommand\labelenumi{%
 \points{enumi}}




\begin{abstract}
Nous construisons des exemples de variétés projectives $\Quo$ proprement convexes de volume fini, non hyperbolique, non compacte en toute dimension $n \geqslant 2$. Ceci nous permet au passage de construire des groupes discrets Zariski-dense de $\ss$ qui ne sont ni des réseaux de $\ss$, ni des groupes de Schottky. De plus, l'ouvert proprement convexe $\O$ est strictement convexe, même Gromov-hyperbolique. 
\end{abstract}

\begin{altabstract}
We build examples of properly convex projective manifold $\Quo$ which have finite volume, are not compact, nor hyperbolic in every dimension $n \geqslant 2$. On the way, we build Zariski-dense discrete subgroups of $\ss$ which are not lattice, nor Schottky groups. Moreover, the open properly convex set $\O$ is strictly-convex, even Gromov-hyperbolic.
\end{altabstract}

\maketitle
\setcounter{tocdepth}{2} 


\mainmatter

\section{Introduction}
\par{
Une variété projective proprement convexe est le quotient d'un ouvert ouvert proprement convexe $\O$ de l'espace projectif réel $\PP^n=\PP^n(\R)$ par un sous-groupe discret sans torsion $\G$ de $\ss$ qui préserve $\O$. Lorsque le quotient $\Quo$ est compact, ces variétés ont été beaucoup étudiés durant ces dernières années. On pourra lire par exemple les articles suivants : \cite{MR2010735,MR2094116,MR2195260,MR2218481,Crampon:2009xy,MR1053346}. Pour un survol de l'état du sujet en 2006, on pourra lire \cite{MR2464391}.
}
\\
\par{
Un ouvert proprement convexe de l'espace projectif réel possède une distance (dite de Hilbert) et une mesure (dite de Busemann) invariantes par les transformations projectives qui le préservent. Nous détaillons ces points au paragraphe \ref{geo_hilbert}, passons plutôt à l'exemple essentiel. 
}
\\
\par{
L'exemple le plus important d'ouvert proprement convexe est l'ellipsoïde. On considère la forme quadratique $q(x_1,....,x_{n+1})=x_1^2+...+x_n^2-x_{n+1}^2$ sur $\R^{n+1}$. On note $\E$ la projection du cône de lumière de $q$ (i.e. l'ensemble des points $\{ x\in \R^{n+1} \,\mid \, q(x) < 0 \}$) sur $\PP^n$. Nous appelerons toute image par une transformation projective de l'ouvert $\E$: un \emph{ellipsoïde}. Muni de sa distance de Hilbert, un ellipsoïde est isométrique à l'espace hyperbolique réel $\H^n$. Il s'agit du modèle projectif de l'espace hyperbolique, que l'on appele parfois modèle de Beltrami-Klein. En particulier, cet ouvert est \emph{homogène}, c'est-à-dire que le groupe $\Aut(\O)= \{ \g \in \ss \, \mid \, \g(\O) = \O\}$ agit transitivement sur $\O$. La figure \ref{pavage} montre un pavage par une tuile compacte et un pavage par une tuile non compacte mais de volume fini du modèle projectif de l'espace hyperbolique. 
}

\begin{figure}[!h]
\begin{center}
\includegraphics[trim=0cm 0cm 0cm 0cm, clip=true, width=6cm]{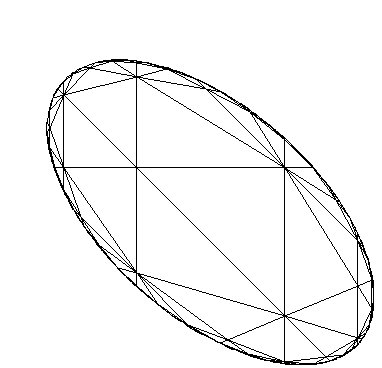}
\reflectbox{\includegraphics[trim=0cm 0cm 0cm 0cm, clip=true, width=6cm]{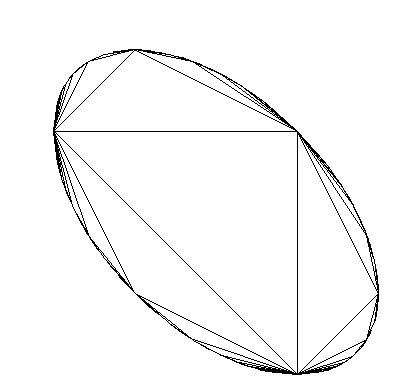}}
\caption{Modèle de Klein-Beltrami de l'Espace Hyperbolique Réel} \label{pavage}
\end{center}
\end{figure}


\par{
L'auteur s'est interessé dans sa thèse à la description des surfaces projectives convexe de volume fini (\cite{Marquis:2009sf,Marquis:2009kq}). Le but de cette article est de montrer le théorème suivant:
}
\begin{thm}
En toute dimension $n \geqslant 2$, il existe un couple $(\O_n,\G_n)$ où $\O_n$ est un ouvert proprement convexe strictement convexe de $\PP^n$ et $\G_n$ un sous-groupe discret de $\ss$ qui préserve $\O_n$ et tel que:
\begin{enumerate}
\item Le quotient $\O_n/_{\G_n}$ est de volume fini. 
\item Le quotient $\O_n/_{\G_n}$ n'est pas compact.
\item Le groupe $\G_n$ est d'indice fini dans le groupe $\Aut(\O)$. En particulier, l'ouvert proprement convexe $\O_n$ n'est pas homogène.
\end{enumerate}
De plus, l'ouvert $\O_n$ que l'on va construire sera Gromov-hyperbolique et le groupe $\G_n$ sera Zariski-dense dans $\ss$.
\end{thm}

\begin{rema}
Yves Benoist a montré dans \cite{MR2094116} que tout ouvert proprement convexe de $\PP^n$ Gromov-hyperbolique est strictement convexe. Karlsson et Noskov ont montré dans \cite{MR1923418} que le bord $\partial \O$ de tout ouvert proprement convexe $\O$ de $\PP^n$ Gromov-hyperbolique est $\C^1$. Yves Benoist a montré dans \cite{MR2010741} des propriétés de régularités encore plus fortes mais plus complexes a énoncé.
\end{rema}

\par{
Avant de faire quelques rappels historiques nous allons rappeler les definitions précises de tous ces objets.
}
\subsection{Géométrie de Hilbert}\label{geo_hilbert}
\par{
Cette partie constitue une introduction très rapide à la géométrie de
Hilbert. Pour une introduction beaucoup plus complète à la géometrie de Hilbert on pourra lire \cite{MR2270228}.
}

\subsubsection{Convexité}
\begin{defi}
Une \emph{carte affine} $A$ de $\PP^n$ est le complémentaire d'un hyperplan projectif. Une carte affine possède une structure naturelle d'espace affine. Un ouvert $\O$ de $\PP^n$ différent de $\PP^n$ est \emph{convexe} lorsqu'il est inclus dans une carte affine et qu'il est convexe dans cette carte. Un ouvert convexe $\O$ de $\PP^n$ est dit \emph{proprement convexe} lorsqu'il existe une carte affine contenant son adhérence $\overline{\O}$. Autrement dit, un ouvert convexe est proprement convexe lorsqu'il ne contient pas de droite affine. Un ouvert convexe $\O$ de $\PP^n$ est dit \emph{strictement convexe} lorsque son bord $\partial \O$ ne contient pas de segment non trivial. 
\end{defi}

\subsubsection{La métrique d'un ouvert proprement convexe}\label{base}
\par{
Soit $\O$ un ouvert proprement convexe de $\PP^n$, Hilbert a
introduit sur de tels ouverts une distance, la distance de
Hilbert, définie de la façon suivante:
}
\\
\par{
Soient $x \neq y \in \O$, on note $p,q$ les points d'intersection
de la droite $(xy)$ et du bord $\partial \O$ de $\O$ tels que $x$
est entre $p$ et $y$, et $y$ est entre $x$ et $q$ (voir figure \ref{dis}). On pose:
}
\begin{figure}[!h]
\begin{center}
\includegraphics[trim=0cm 12cm 0cm 0cm, clip=true, width=6cm]{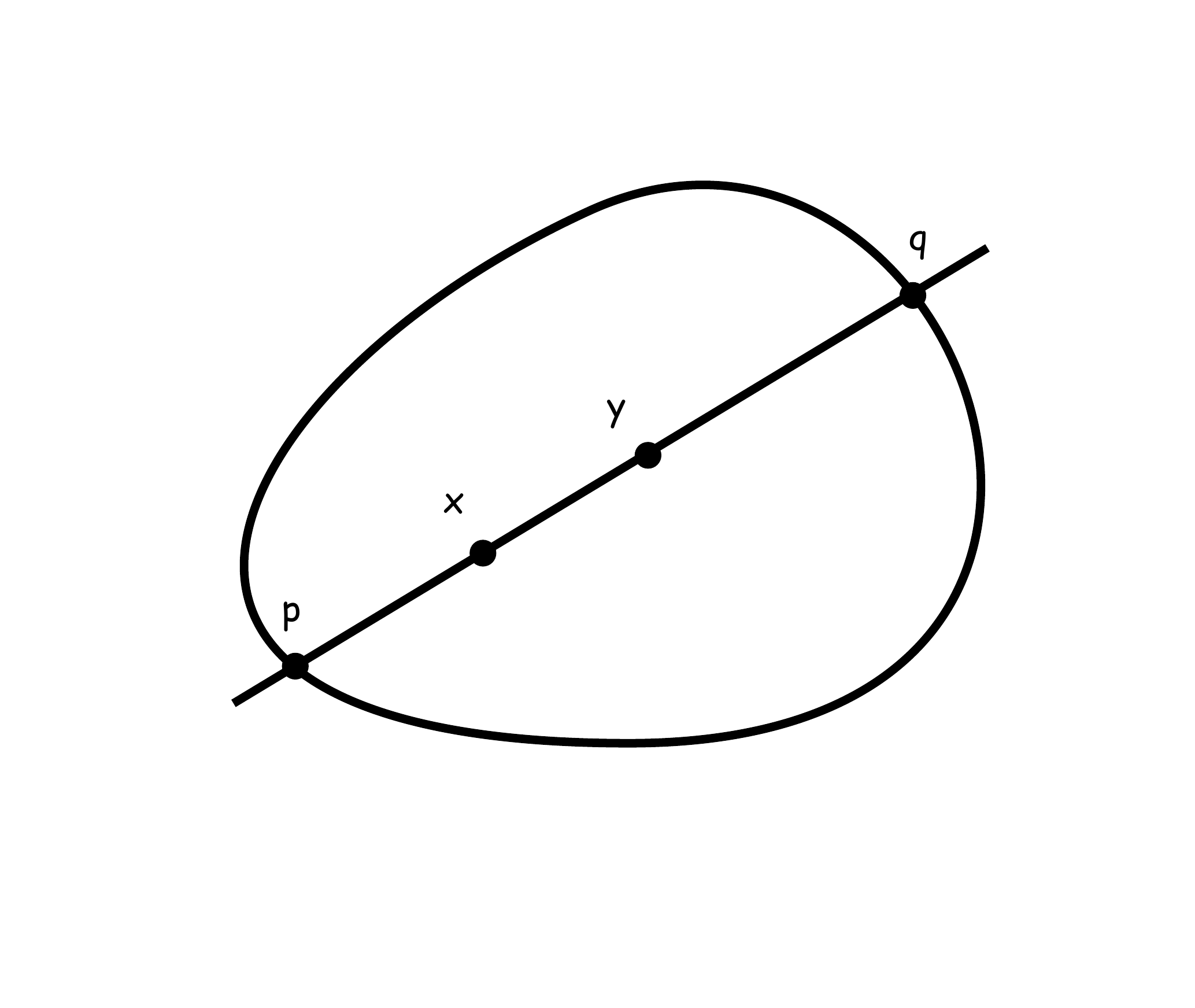}
\caption{La distance de Hilbert} \label{dis}
\end{center}
\end{figure}

$$
\begin{array}{ccc}
d_{\O}(x,y) = \frac{1}{2}\ln ([p:x:y:q]) = \frac{1}{2}\ln \Big(\frac{\|p-y \|\cdot \|
q-x\|}{\| p-x \| \cdot \| q-y \|} \Big) & \textrm{et} &
d_{\O}(x,x)=0
\end{array}
$$
\begin{enumerate}
\item $[p:x:y:q]$ désigne le birapport des points $p,x,y,q$.

\item $\| \cdot \|$ est une norme euclidienne quelconque sur une
carte affine $A$ qui contient l'adhérence $\overline{\O}$ de $\O$.
\end{enumerate}

\begin{rema}
Il est clair que $d_{\O}$ ne dépend ni du choix de $A$, ni du
choix de la norme euclidienne sur $A$.
\end{rema}

\begin{fait}
Soit $\O$ un ouvert proprement convexe de $\PP^n$,
\begin{enumerate}
\item $d_{\O}$ est une distance sur $\O$.

\item $(\O,d_{\O})$ est un espace métrique complet.

\item La topologie induite par $d_{\O}$ coïncide avec celle
induite par $\PP^n$.

\item Le groupe $\Aut(\O)$ des transformations projectives de $\ss$ qui préservent $\O$ est un sous-groupe fermé de $\ss$ qui agit par isométrie sur $(\O,d_{\O})$. Il agit donc proprement sur $\O$.
\end{enumerate}
\end{fait}

\subsubsection{La structure finslérienne d'un ouvert proprement convexe}

\par{
Soit $\O$ un ouvert proprement convexe de $\PP^n$, la métrique de
Hilbert $d_{\O}$ est induite par une structure finslérienne sur
l'ouvert $\O$. On identifie le fibré tangent $T \O$ de $\O$ à
$\O \times A$.
}
\\
\par{
Soient $x \in \O$ et $v \in A$, on note $p^+$ (resp. $p^-$) le
point d'intersection de la demi-droite définie par $x$ et $v$
(resp $-v$) avec $\partial \O$.
}
\\
\par{
On pose: $\|v\|_x = \frac{1}{2}\Big(\frac{1}{\|x-p^-\|} + \frac{1}{\| x-
p^+\|} \Big) \| v \|$.
}
\begin{figure}[!h]
\begin{center}
\includegraphics[trim=0cm 12cm 0cm 0cm, clip=true, width=6cm]{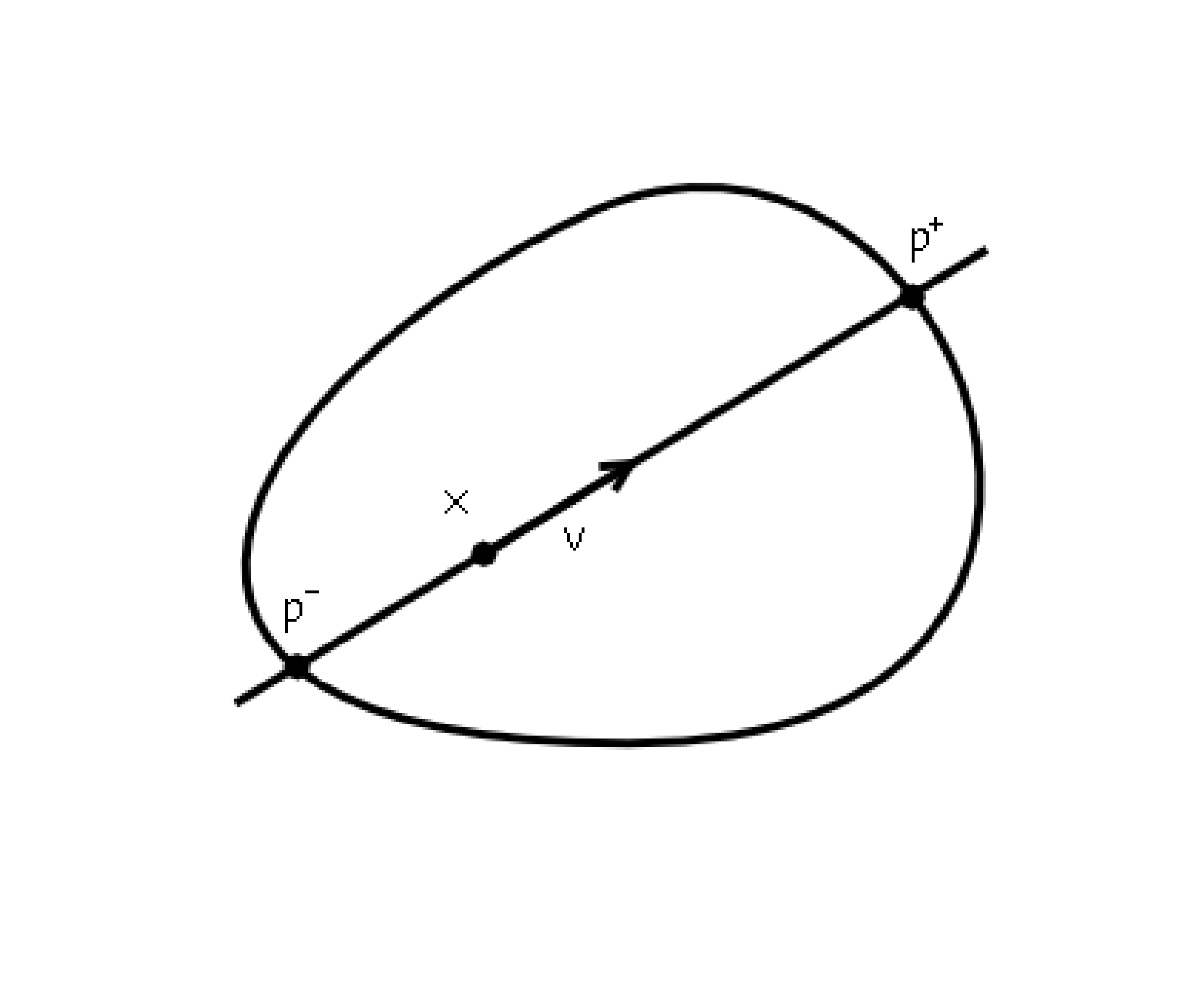}
\caption{La métrique de Hilbert} \label{met}
\end{center}
\end{figure}

\begin{fait}
Soient $\O$ un ouvert proprement convexe de $\PP^n$ et $A$ une carte
affine qui contient $\overline{\O}$,
\begin{enumerate}
\item la distance induite par la métrique finslérienne $\|\cdot \|
_{\cdot}$ est la distance $d_{\O}$.

\item Autrement dit on a les formules suivantes:
\begin{itemize}
\item $\|v\|_x = \frac{d}{dt}|_{t=0} d_{\O}(x,x+tv)$, où $v \in
A$, $t\in \R$ assez petit.

\item $d_{\O}(x,y) = \inf \int_0^1 \| \sigma'(t) \|_{\sigma(t)}
dt$, où $l'\inf$ est pris sur les chemins $\sigma$ de classe
$\C^1$ tel que $\sigma(0)=x$ et $\sigma(1)=y$.
\end{itemize}
\end{enumerate}
\end{fait}

\begin{rema}
La quantité $\|v\|_x$ est donc indépendante du choix de $A$ et de $\| \cdot
\|$.
\end{rema}

\subsubsection{Mesure sur un ouvert proprement convexe}

\par{
Nous allons construire une mesure borélienne $\mu_{\O}$ sur $\O$,
de la même façon que l'on construit une mesure borélienne sur une
variété riemanienne.
}

Soit $\O$ un ouvert proprement convexe de $\PP^n$, on note:
\begin{itemize}
\item $B_x(1) = \{ v \in T_x \O \, | \, \|v\|_x < 1 \}$

\item $\textrm{Vol}$ est la mesure de Lebesgue sur $A$ normalisée
pour avoir $\textrm{Vol}(\{ v \in A \, | \, \|v\| < 1 \})=1$.
\end{itemize}

\par{
On peut à présent définir la mesure $\mu_{\O}$. Pour tout borélien
$\mathcal{A} \subset \O \subset A$, on pose:

$$\mu_{\O} (\mathcal{A})= \int_{\mathcal{A}} \frac{dVol(x)}{\textrm{Vol}(B_x(1))}$$

La mesure $\mu_{\O}$ est indépendante du choix de $A$ et de $\| \cdot \|$,
car c'est la mesure de Haussdorff de $(\O,d_{\O})$.
}

\begin{defi}
Le quotient d'un ouvert proprement convexe $\O$ par un sous-groupe discret $\G$ hérite d'une mesure via la mesure de Busemann de $\O$. On dira que $\Quo$ est de \emph{volume fini} lorsqu'il est de volume fini pour cette mesure.
\end{defi}

\subsection{Petit historique autour de la construction d'ouverts proprement convexes possédant "beaucoup de symétries"}

\subsubsection{Le cas homogène}
\par{
Une première définition d'ouverts proprement convexes possédant "beaucoup de symétries" est un ouvert proprement convexe \emph{homogène}, c'est à dire tel que le groupe $\Aut(\O)$ agit transitivement sur $\O$. Koecher et Vinberg ont classifié ces ouverts dans les années 50-60 dans les deux articles suivants: \cite{MR0201575,MR0158414}.
}
\\
\par{
 La liste de ces ouverts est assez longue et ne nous intéresse pas car un ouvert proprement convexe homogène $\O$ possède un sous-groupe discret $\G \subset \Aut(\O)$ tel que le quotient $\Quo$ est de volume fini si et seulement si le groupe $\Aut(\O)$ est unimodulaire si et seulement si l'ouvert $\O$ est un espace symétrique (i.e il existe une symétrie centrale centrée en n'importe quel point).
}
\\
\par{
Les ouverts proprement convexes qui nous intéressent sont sont ceux des deux définitions suivantes:
}

\begin{defi}
Un ouvert proprement convexe $\O$ de $\PP^n$ est dit \emph{divisible} lorsqu'il existe un sous-groupe discret $\G$ de $\ss$ tel que $\G \subset \Aut(\O)$ et $\Quo$ est compact.
\end{defi}

\begin{defi}
Un ouvert proprement convexe $\O$ de $\PP^n$ est dit \emph{quasi-divisible} lorsqu'il existe un sous-groupe discret $\G$ de $\ss$ tel que $\G \subset \Aut(\O)$ et $\Quo$ est de volume fini.
\end{defi}

\par{
Les ellipsoïdes forment (avec $n \geqslant 1$) sont les seules
convexes divisibles (resp. quasi-divisibles) homogènes et strictement convexes.
}
\\
\par{
Il existe des convexes divisibles (resp. quasi-divisibles) homogènes et non strictement convexe. Voici la liste des irréductibles avec $n \geqslant 2$:
\begin{enumerate}
\item $\Pi_n(\R)=\pi($ \{ Les matrices réelles $(n+1) \times (n+1)$ symétriques définies
positives) \}, il est de dimension $m=\frac{(n-1)(n+2)}{2}$ et son
groupe d'automorphisme est $\ss$.

\item $\Pi_n(\mathbb{C}) =\pi($ \{Les matrices complexes $(n+1) \times (n+1)$
hermitiennes définies $\,\,\,\,\,\,\,\,\,\,\,$ positives \}), il est de dimension $m=n^2-1$
et son groupe d'automorphisme est $\mathrm{SL}_{n+1}(\mathbb{C})$.

\item $\Pi_n(\mathbb{H}) =\pi($ \{ Les matrices quarternioniques $(n+1) \times (n+1)$
hermitiennes définies positives \}), il est de dimension
$m=(2n+1)(n-1)$ et  son groupe d'automorphisme est
$\mathrm{SL}_{n+1}(\mathbb{H})$.

\item $\Pi_3(\mathbb{O})$ un convexe "exceptionnel" de dimension
26 et tel que
Lie($\Aut(\Pi_3(\mathbb{O})))=\mathfrak{e}_{6(-26)}$.
\end{enumerate}
}
\vspace{0.5cm}

\subsubsection{Le cas non-homogène}

\par{
Kac et Vinberg ont construit les premiers exemples de convexe divisible non-homogène dans \cite{MR0208470} à l'aide de groupe de Coxeter. Les résultats joints de Johnson et Millson (\cite{MR900823}), Koszul (\cite{MR0239529}) et Benoist (\cite{MR2094116}) montrent qu'en toute dimension $n \geqslant 2$, il existe des convexes divisibles non homogènes et strictement convexes.
}
\\
\par{
 Kapovich et Benoist ont construit en toute
dimension $n \geqslant 4$ (Benoist pour $n=4$ dans \cite{MR2295544} et
Kapovich pour $n \geqslant 4$ dans \cite{MR2350468}) des convexes divisibles non homogènes, strictement convexes et non quasi-isométriques à l'espace hyperbolique $\mathbb{H}^n$.
}
\\
\par{
Dans \cite{Marquis:2009kq}, l'auteur a montré que tout convexe quasi-divisible de dimension 2 est strictement convexe. La généralisation de ce résultat en dimension supérieure est fausse. En effet, Yves Benoist a construit des exemples de convexe divisible irréductible, non homogène et non strictement convexe en dimension 3, 4, 5 et 6 (\cite{MR2218481}). Cette famille de convexe divisible est la plus difficile à construire. Les constructions de l'article \cite{Marquis:2008qf} devrait permettre de construire des exemples de convexe quasi-divisible irréductible, non homogène et non strictement convexe en dimension 3.
}

\subsection{Espaces des modules}
\par{
On rappelle dans ce paragraphe les définitions de structure projective, d'espaces des modules de structures projectives, etc... pour éviter les ambiguïtés.
}

\begin{defi}
Une \emph{structure projective réelle} sur une variété $M$ est la donnée d'un atlas maximal $\varphi_{\mathcal{U}}:\mathcal{U} \rightarrow \PP^n$ sur $M$ tel que les fonctions de transitions $\varphi_{\mathcal{U}} \circ \varphi_{\mathcal{V}}^{-1}$ sont des éléments de $\ss$, pour tous
ouverts $\mathcal{U}$ et $\mathcal{V}$ de l'atlas de $S$ tel que $\mathcal{U} \cap \mathcal{V} \neq \varnothing$.
\end{defi}

\begin{rema}
Pour simplifier la rédaction on dira \og structure projective\fg $\,$ à la place de \og structure projective réelle \fg.
\end{rema}

\begin{defi}
Un \emph{isomorphisme} entre deux variétés munies de structures projectives est un homéomorphisme qui, lu dans les cartes, est donné par des éléments de $\ss$.
\end{defi}

\noindent La donnée d'une structure projective sur une variété $M$ est équivalente à la donnée:
 \begin{enumerate}
 \item d'un homéomorphisme local $\dev:\widetilde{M}\rightarrow \PP^n$ appelée \emph{développante}, où $\widetilde{M}$ est le revêtement universel de $M$,

 \item d'une représentation $\hol:\pi_1(M) \rightarrow \ss$ appelée \emph{holonomie} tel que la développante est $\pi_1(M)$-équivariante ( i.e pour tout $x \in \widetilde{M}$, et pour tout $\g \in \pi_1(M)$ on a $\dev(\g \, x) = \hol(\g) \dev(x)$).
\end{enumerate}
De plus, deux structures projectives données par les couples $(\dev,\hol)$ et $(\dev',\hol')$ sont isomorphes si et seulement s'il existe un élément $g \in \ss$ tel que $\dev' = g \circ \dev$ et $\hol' = g \circ \hol \circ g^{-1}$.

\begin{defi}
Soit $M$ une variété, une \emph{structure projective marquée sur $M$} est la donnée d'un homéomorphisme $\varphi:M \rightarrow \M$ où
$\M$ est une variété projective. On note $\PP'(M)$ \emph{l'ensemble des structures projectives marquées
sur $M$}.

Deux structures projectives marquées sur $M$, $\varphi_1:M \rightarrow \M_1$ et $\varphi_2:M \rightarrow \M_2$ sont dites
\emph{isotopiques} lorsqu'il existe un isomorphisme $h:\M_1 \rightarrow \M_2$ tel que $\varphi_2^{-1} \circ h \circ \varphi_1 : M
\rightarrow M$ est un homéomorphisme \underline{isotope à l'identité}. On note $\PP(M)$ \emph{l'ensemble des structures projectives marquées sur $M$ modulo isotopie}.
\end{defi}

On peut à présent définir une topologie sur l'ensemble des structures projectives marquées sur la variété $M$. On introduit l'espace:

$$
\mathcal{D'}(M)=
\left\{
\begin{array}{l|l}
          & \dev:\widetilde{\M}\rightarrow \PP^n \textrm{ est un homéomorphisme local}\\
(\dev,\hol) & \hol:\pi_1(M) \rightarrow \ss\\
          & \dev \textrm{ est } \pi_1(M)-\textrm{équivariante}
\end{array}
\right\}
$$

\par{
Les espaces $\widetilde{\M}$, $\PP^n$, $\pi_1(M)$ et $\ss$ sont des espaces topologiques localement compacts. On munit l'ensemble des applications continues entre deux espaces localement compacts de la topologie compact-ouvert. Ainsi, l'espace $\mathcal{D}'(M)$ est munie d'une topologie. Le groupe $Homeo_0(M)$ des homéomorphismes isotopes à l'identité agit naturellement sur $\mathcal{D}'(M)$. Le groupe $\ss$ agit aussi naturellement sur $\mathcal{D}'(M)$. Ces deux actions commutent. L'espace quotient est l'espace $\PP(M)$ des structures projectives marquées sur $M$ à isotopie près. On le munit de la topologie quotient.
}
\\
\par{
On ne s'intéresse qu'à un certain type de structure projective: les structures projectives proprement convexes.
}
\begin{defi}
Soit $M$ une variété, une structure projective sur $M$ est dite
\emph{proprement convexe} lorsque la développante est un
homéomorphisme sur un ouvert $\Omega$ proprement convexe de $\PP^n$. On note
$\beta(M)$ l'ensemble des structures projectives proprement
convexes sur $M$ modulo isotopie.
\end{defi}

\par{
Soit $\M$ une variété projective proprement convexe, l'application développante permet d'identifier le revêtement universel $\widetilde{\M}$ de $\M$ à un ouvert $\O$ proprement convexe de $\PP^n$  qui est naturellement muni d'une mesure $\mu_{\O}$ invariante sous l'action du groupe fondamental $\pi_1(\M)$ de $\M$. On note $\pi:\O \rightarrow \M$ le revêtement universel de $\M$. Il existe une unique mesure $\mu_{\M}$ sur $\M$ telle que pour tout borélien $\mathcal{A}$ de $\O$, si $\pi:\O \rightarrow \M$ restreinte à $\mathcal{A}$ est injective alors $\mu_{\M}(\pi(\mathcal{A}))=\mu_{\O}(\mathcal{A})$.
}

\begin{defi}
Soit $M$ une variété, on dit qu'une structure projective 
proprement convexe $\M$ sur $M$ est de \emph{volume fini} lorsqu'on a
$\mu_{\M}(\M) < \infty$. On note $\beta_f(M)$ \emph{l'espace des modules des structures projectives marquées proprement convexes de volume fini sur $M$}.
\end{defi}

\subsection{Plan}

\par{
Nous allons à présent expliquer la structure de cet article. Dans la partie \ref{constr_depa}, nous rappelons rapidement comment construire une variété hyperbolique qui possède une hypersurface totalement géodésique.
}
\\
\par{
Dans la partie \ref{sec_pliage}, nous rappelons comment plier une variété hyperbolique $\M$ le long d'une hypersurface totalement géodésique $\Nn$. Le pliage est une déformation $\rho_t: \pi_1(M) \rightarrow \ss$ non triviale de la représentation $\rho_0: \pi_1(M) \rightarrow \SO$ du groupe fondamental de la variété hyperbolique $\M$. Ce type de déformation a été utilisée par Johnson et Millson dans \cite{MR900823}, sous le nom de "bending". Le pliage nous permet d'obtenir une nouvelle structure projective sur $M$ qui n'est plus hyperbolique. La première difficulté consiste à montrer que cette nouvelle structure est encore PROPREMENT CONVEXE. Il faut montrer qu'il existe un ouvert $\O_t$ préservé par $\rho_t$. Pour cela, nous utiliserons un théorème de convexité qui sera présenté et démontré dans la partie \ref{sec_conv}.
}
\\
\par{
Dans la partie \ref{sec_zar}, nous montrerons que les groupes $\rho_t(\pi_1(M))$ ainsi construit sont irréductibles, Zariski-dense dans $\ss$ (pour $t \neq 0$) et que leur action sur le bord de $\O_t$ est minimale. On en profitera pour montrer que le pliage est bien une déformation non triviale, que les groupes obtenus ne sont pas des réseaux de $\ss$ et que les ouverts proprement convexes obtenus ne sont pas homogènes.
}
\\
\par{
Dans la partie \ref{sec_volfini}, nous montrons que l'action de $\rho_t(\pi_1(M))$ sur $\O_t$ est de covolume fini. Ceci nous fournira un argument pour montrer que les groupes obtenus ne sont pas des groupes de Schottky.
}
\\
\par{
Enfin dans la partie \ref{sec_gro}, nous montrerons que l'ouvert $\O_t$ est strictement convexe et même Gromov-hyperbolique.
}
\\
\par{
L'auteur tient à remercier vivement les laboratoires de l'UMPA à Lyon et de TIFR à Mumbaï pour les extraordinaires conditions de travail qu'ils offrent. Plus particulièrement, l'auteur remercie Venkataramana, Yves Benoist, Nicolas Bergeron et Constantin Vernicos pour leurs aides à distance ou en "live". Et plus particulièrement, Mickaël Crampon pour son aide pour les démonstrations de la partie \ref{sec_gro}.
}


\section{Construction de variétés hyperboliques possédant une hypersurface totalement géodésiques}\label{constr_depa}

La proposition suivante est très classique. Elle sert par exemple de point de départ pour construire des variétés hyperboliques non arithmétiques (\cite{MR932135}), ou avec un premier nombre de Betti arbitrairement grand (Voir par exemple \cite{MR1769939}).

\begin{prop}\label{exisreseau}
En toute dimension $n \geqslant 2$, il existe une variété hyperbolique de volume fini non compacte $M_n$ et  une variété hyperbolique compacte $M'_n$ qui possède une hypersurface totalement géodésique de volume fini (pour la métrique hyperbolique induite).
\end{prop}

L'objet de cette partie est de rappeller les grandes lignes de la démonstration de cette proposition dans le cas non compacte.

\subsection{Sous-variété totalement géodésique immergée et plongée des variétés hyperboliques}

\par{
Nous aurons besoin du théorème suivant. Ce théorème a une longue histoire et de nombreux auteurs, on pourra trouver une preuve de la version qui nous intéresse dans l'article de \cite{MR1769939} (Théorème $1'$) de Nicolas Bergeron. Ce théorème nous dit que si l'on peut immerger de façon totalement géodésique une variété hyperbolique "proprement" alors on peut la plonger quitte à passer un revêtement fini. Voici l'énoncé précis.
}

\begin{theo}\label{relevement}
Soit $M$ une vari\'et\'e hyperbolique de volume fini de dimension $n$ et $N$ une vari\'et\'e hyperbolique de dimension $n-1$ de volume fini. On suppose qu'il existe une immersion propre totalement g\'eod\'esique $\varphi$ de $N$ dans $M$ alors il existe un rev\^etement fini $\widehat{M}$ de $M$ et un rev\^etement fini $\widehat{N}$ de $N$ tel que le relev\'e $\widehat{\varphi}:\widehat{N} \rightarrow \widehat{M}$ de $\varphi: N \rightarrow M$ soit un plongement totalement g\'eod\'esique.
\end{theo}

\subsection{Construction de la variété hyperbolique $M_n$}

\par{
Dans ce paragraphe, on construit la variété $M_n$ de la proposition \ref{exisreseau}.
Pour construire une telle variété, nous allons utiliser le théorème de Borel et Harish-Chandra 
qui permet de construire des réseaux arithmétiques uniforme et non uniforme dans les groupes de Lie semi-simple. On donne ici une version très simplifiée de ce théorème dans le cas où le groupe de Lie est $\SO$ et le réseau est non uniforme.
}

\begin{theo}[Borel et Harish-Chandra]\label{Borel}
Soit $q$ une forme quadratique sur $\R^{n+1}$ à coefficients dans $\Q$, avec $n \geqslant 2$. On suppose que:
\begin{enumerate}
\item La forme quadratique $q$ représente $0$ sur $\Q$.
\item La forme quadratique $q$ est de signature $(n,1)$.
\end{enumerate}
Alors, le groupe  $\mathrm{SL}_{n+1}(\Z) \cap  \mathrm{SO}(q)$ est un r\'eseau non-uniforme de $\mathrm{SO}(q)$.
\end{theo}

Nous aurons aussi besoin du lemme de Selberg.

\begin{theo}[Selberg]\label{selberg}
Tout sous-groupe de type fini de $\mathrm{GL}_{n}(\mathbb{C})$ est virtuellement sans torsion.
\end{theo}

\par{
Rappellons qu'un groupe $\G$ est virtuellement sans torsion s'il contient un sous-groupe d'indice fini $\G'$ qui est sans torsion, (c'est à dire que tous les éléments de $\G'$ sont d'ordre infini).
}
\\
\par{
On considère la forme quadratique $q(x_1,....,x_{n+1})=x_1^2+...+x_n^2-x_{n+1}^2$, elles vérifient les hypothèses du théorème \ref{Borel}. Par conséquent, le groupe $\Lambda_1 =\mathrm{SL}_{n+1}(\Z) \cap  \mathrm{SO}(q)$ est un réseau non-uniforme de $\SO$. Le lemme de Selberg montre qu'il existe un sous-groupe d'indice fini $\Lambda_2 \subset \Lambda_1$ sans torsion. 
}
\\
\par{
Le groupe $\Delta_2$ des éléments de $\Lambda_2$ qui préservent l'hyperplan $H=\{ x_1=0 \}$ est un sous-groupe d'indice fini du groupe $\mathrm{SL}_{n}(\Z) \cap  \mathrm{SO}(q')$, où $q'(x_2,....,x_{n+1})=x_2^2+...+x_n^2-x_{n+1}^2$ , c'est donc un réseau de $\mathrm{SO}(q')$ (uniforme si $n=2$ et non uniforme si $n\geqslant 3$).
}
\\
\par{
L'application naturelle $\varphi$ de $H/_{\Delta_2}$ vers $\H^n/_{\Lambda_2}$ est une immersion totalement géodésique propre. Le théorème \ref{relevement} montre qu'il existe un sous-groupe $\Delta_0$ (resp. $\Lambda_0$) d'indice fini de $\Delta_2$ (resp. $\Lambda_2$) tel que le relèvement associé de $\varphi$ est un plongement.
}
\\
\par{
Par conséquent, la variété $M_n = \H^n/_{\Lambda_0}$ est une variété hyperbolique de volume fini qui possède une hypersurface $\Nn_0=H/_{\Delta_0}$ totalement géodésique de volume fini. 
}

\begin{nota}
Tout au long de ce texte, le symbole $\Lambda_0$ (resp.  $M_n$) désignera le groupe $\Lambda_0$  (resp. la \underline{variété topologique} $M_n$) que l'on vient de construire. On désignera par le symbole $\M_0$ la structure hyperbolique que l'on vient de construire sur $M_n$. On désignera par le symbole $\Nn_0$ l'hypersurface totalement géodésique de $\M_0$ que l'on vient de construire.
\end{nota}

\section{Pliage}\label{sec_pliage}

Nous allons à présent contruire une famille continue de structures projectives sur la variété topologique $M_n$.

\subsection{Présentation}

\begin{defi}
Soit $M$ une variété. Soit $\mathcal{M}$ une structure projective sur $M$. Une \emph{déformation projective} de $\mathcal{M}$ est un chemin continu $c:\R  \rightarrow  \PP(M)$ tel que  $c(0) = \mathcal{M}$. Une déformation est dite \emph{triviale} lorsque le chemin $c$ est constant.
\end{defi}

Johnson et Millson ont montré le théorème suivant dans \cite{MR900823}.

\begin{theo}[Johnson-Millson]\label{johMill}
Soit $\M$ une variété hyperbolique. Si $\M$ possède une hypersurface totalement géodésique $\Nn$ alors il existe une déformation projective non triviale de $\M$.
\end{theo}

\par{
Comme nous allons utiliser la même déformation que celle introduite par Johnson et Millson. Nous allons dans le paragraphe \ref{defor_proj} qui suit reprendre la construction de cette déformation. Nous ne montrerons pas dans le paragraphe \ref{defor_proj} que cette déformation est effectivement non triviale. Nous le montrerons à l'aide du corollaire \ref{coro_zar}.
}
\\
\par{
Le théorème \ref{theo_kapo} montrera que la structure projective déformée est encore proprement convexe. Ceci entrainera en particulier que l'holonomie de la structure projective déformée est encore fidèle et discrète. Nous donnons une courte démonstration de ce résultat dans l'appendice \ref{sec_conv}. Dans la partie \ref{sec_volfini} nous montrons que la structure projective déformée est de volume fini.
}

\subsection{Déformation de structure projective}\label{defor_proj}

\par{
Soit $\M$ une variété hyperbolique de volume fini et $\Nn$ une hypersurface totalement géodésique de $\M$ de volume fini. On note $\dev_0: \widetilde{\M} \rightarrow \O \subset \PP^n$ et $\rho_0 : \pi_1(M) \rightarrow \ss$ un couple développante-holonomie qui définit la structure projective $\M$. On note $H$ une composante connexe de la préimage de $\Nn$ dans $\O$. On se donne $x_0$ un point de $H$. La variété $\M-\Nn$ possède une ou deux composantes connexes, nous allons distinguer ces cas dans les deux paragraphes suivants.
}

\subsubsection{Déformation dans le cas séparant}

\par{
On suppose que la variété $\M-\Nn$ possède deux composantes connexes $\M_g$ et $\M_d$.
}
\\
\par{
Le théorème de Van Kampen montre que le groupe fondamental de $\M$ peut s'écrire comme le produit amalgamé suivant:
}

$$
\begin{array}{cccc}
\pi_1(\M) = &   \pi_1(\M_g) & *                  &  \pi_1(\M_d)  \\
                  &    \multicolumn{3}{c}{\pi_1(\Nn)}    \\
\end{array}
$$

\par{
On cherche  à déformer la représentation $\rho_0$, pour cela on peut essayer de définir une nouvelle représentation de la façon suivante:
}
\par{
Soit $a \in \ss$, on pose:
}
$$
\begin{array}{ccclc}
\rho_a :  &  \pi_1(\M) & \rightarrow  & \ss &  \\ 
               &    \g         & \mapsto      & \left\{  \begin{array}{lc} 
                                                                            \rho_0(\g)                  & \textrm{si } \g \in \pi_1(\M_g) \\ 
                                                                            a \rho_0(\g)  a^{-1}    & \textrm{si } \g \in \pi_1(\M_d) \\
                                                                            \end{array}
                                                                            \right.
\end{array} 
$$

\par{
La représentation $\rho_a$ est bien définie si et seulement si $\forall \delta \in \pi_1(\Nn)$, on a $\rho_0(\delta)  =  a \rho_0(\delta)  a^{-1} $, autrement dit si et seulement si $a$ appartient au centralisateur de $\rho_0(\pi_1(\Nn))$ dans $\ss$. L'existence d'un tel élément est assurée par le lemme \ref{centra}.
}
\\
\par{
$\C_g$ et $\C_d$ les adhérences des deux composantes connexes de $\O\,\, - \underset{\g \in \pi_1(M)}{\bigcup} \g H$ (les $\g H$ sont disjoints puisque $\Nn$ est une hypersurface) qui bordent $H$. Le stabilisateur de $\C_g$ (resp. $\C_d$) dans $\pi_1(\M)$ est le groupe $\pi_1(\M_g)$ (resp. $\pi_1(\M_d)$).
}
\\
\par{
 La nouvelle développante est l'unique homéomorphisme local $\rho_a$-équivariant qui prolonge l'application $\dev_a:\C_g \cup \C_d \rightarrow \PP^n$ suivante:
}

\begin{enumerate}
\item Si $x \in \C_g$ alors on pose $\dev_a(x)= \dev_0(x)$.

\item Si $x \in \C_d$ alors on pose $\dev_a(x)= a \cdot \dev_0(x)$.

\item L'existence et l'unicité du prolongement de $\dev_a$ à $\O=\widetilde{\M}$ est évidente.
\end{enumerate}

\par{
Le théorème \ref{theo_kapo} montrera que $\dev_a$ est un homéomorphisme sur un ouvert proprement convexe de $\O$.
}

\subsubsection{Déformation de structure projective dans le cas non séparant}

\par{
On suppose que la variété $\M-\Nn$ est connexe.
}
\\
\par{
On note $\Mc$ la variété à bord obtenue en découpant $\M$ le long de $\Nn$, c'est à dire en ajoutant deux copies de $\Nn$ à $\M-\Nn$. La variété $\Mc$ possède deux bords $\Nn_g$ et $\Nn_d$. On choisit un point $x_0 \in \Nn$. On a une projection naturelle $p:\Mc \rightarrow \M$ qui est un homéomorphisme lorsqu'on la restreint à l'intérieur de $\Mc$. On se donne $\alpha$ un chemin de la variété $\Mc$ qui va du bord $\Nn_g$ au bord $\Nn_d$ dont la projection $p(\alpha)$ sur $\M$ est un lacet de $\M$ basé en $x_0$. Le théorème de Van Kampen montre que le groupe fondamental de $\M$ peut s'écrire comme la HNN-extension suivante:

$$
\pi_1(\M,x_0) = \pi_1(\Mc,x_g) *_{\alpha}
$$

On note $x_g$ (resp. $x_d$) le point de départ (resp. d'arrivée) de $\alpha$, c'est l'unique point de $\Nn_g$ (resp. $\Nn_d$) qui se projete sur $x_0$. Le fait que $\pi_1(\M) = \pi_1(\Mc) *_{\alpha}$ signifie que $\pi_1(\M)$ est le quotient du produit libre du groupe $\pi_1(\Mc)$ et du groupe engendré par le lacet $\alpha$ par la relation suivante:

$$\forall \g_g \in \pi_1(\Nn_g,x_g), \forall \g_d \in \pi_1(\Nn_d,x_d)  \textrm{ tel que } p^*(\g_g) = p^*(\g_d) \textrm{ alors } \g_g = \alpha^{-1} \g_d \alpha $$
 
On peut essayer de définir une nouvelle représentation de la façon suivante:
Soit $a \in \ss$, on pose:

$$
\begin{array}{cccl}
\rho_a :  &  \pi_1(\M,x_0) & \rightarrow  & \ss  \\ 
              &    \g                & \mapsto      &  \left\{
                                							\begin{array}{cl}
																\rho_0(\g)             & \textrm{si } \g \in \pi_1(\Mc,x_g)\\ 
																 a \rho_0(\alpha)   & \textrm{si } \g= \alpha \\ 
															\end{array}
															\right.
															\end{array}
$$
 
La représentation $\rho_a$ est bien définie si et seulement si:

$\forall \g_g \in \pi_1(\Nn_g,x_g), \forall \g_d \in \pi_1(\Nn_d,x_d)$ tel que $p^*(\g_g) = p^*(\g_d)$ on a:

$$ \rho_a(\g_g) = \rho_a(\alpha^{-1} \g_d \alpha)$$

Or,
\begin{enumerate}
\item $\rho_a(\g_g) = \rho_0(\g_g) = \rho_0(\alpha)^{-1} \rho_0(\g_d)  \rho_0(\alpha)$
\item $\rho_a(\alpha^{-1} \g_d \alpha) = \rho_a(\alpha^{-1}) \rho_a(\g_d) \rho_a(\alpha)= \rho_0(\alpha)^{-1} a^{-1} \rho_0(\g_d) a \rho_0(\alpha)$
\end{enumerate}

\par{
Autrement dit $\rho_a$ est bien définie si et seulement si $a$ appartient au centralisateur de $\rho_0(\pi_1(\Nn_d,x_d)) = \rho_0(\pi_1(\Nn,x_0))$ dans $\ss$. L'existence d'un tel élément est assurée par le lemme \ref{centra}.
}
\\
\par{
On note $\C$ une composante connexe de $\O - \underset{\g \in \pi_1(\M)}{\bigcup}\g H$. Le stabilisateur de $\C$ dans $\pi_1(\M)$ est le groupe $\pi_1(\Mc)$. 
}
\\
\par{
La nouvelle développante $\dev_a$ est l'unique homéomorphisme local $\rho_a$-équivariant qui prolonge l'application $\dev_a|_{\C}= \dev_0|_{\C}$.
}
\\
\par{
Le théorème \ref{theo_kapo} montrera que $\dev_a$ est un homéomorphisme sur un ouvert proprement convexe de $\O$.
}

\subsubsection{Centralisateur du groupe fondamental d'une hypersurface}

Le lemme suivant est élémentaire.

\begin{lemm}\label{centra}
Soit $\Delta$ un réseau de $\mathrm{SO}_{n-1,1}(\R)$, on considère la représentation $\rho$ de $\Delta$ obtenue à l'aide de l'injection qui préservent la première coordonnée de  $\mathrm{SO}_{n-1,1}(\R)$ dans $\SO$ puis de l'injection canonique dans $\ss$. La composante connexe du centralisateur de $\rho(\Delta)$ dans $\ss$ est le groupe de matrices diagonales suivant (pour $t>0$):
$$
a_t =
\left(
\begin{array}{llll}
 e^{nt}    & 0               & \cdots & 0 \\
0            &     e^{-t}   &  \ddots       &       \vdots  \\
\vdots    &  \ddots    &  \ddots &   0 \\
0             &     \cdots            &  0     &  e^{-t} 
\end{array}
\right)
$$
En particulier, la composante connexe du centralisateur de $\rho(\Delta)$ dans $\ss$ est isomorphe à $\R_+^*$.
\end{lemm}

\begin{proof}
Borel a montré que tout réseau d'un groupe de Lie algébrique semi-simple est Zariski dense (\cite{MR0123639}). Par suite, le centralisateur de $\rho(\Delta)$ dans $\ss$ est égale au centralisateur de $\mathrm{SO}_{n-1,1}(\R)$ dans $\ss$. D'où le résultat.
\end{proof}

Il est temps à présent de donner un nom à cette déformation.

\begin{defi}
Soient $\M$ une variété hyperbolique et $\Nn$ une hypersurface totalement géodésique de $\M$ de volume fini. On note $M$ la variété topologique sous-jacente à $\M$. Le chemin $c:\R \rightarrow \PP(M)$ donné par le couple $(\dev_{a_t},\rho_{a_t})$ s'appelle un \emph{pliage de $\M$ le long de $\Nn$}.
\end{defi}

\subsection{Le pliage en terme de cartes}

\par{
Dans ce paragraphe on se propose de donner la définition du pliage en terme de cartes. Commençons par rappeler comment on recolle deux variétés à bord le long d'un bord d'un point de vue topologique. 
}
\\
\par{
Soit $M^{\dagger}$ une variété à bord (à priori non connexe) avec deux bords connexes $N_g$ et $N_d$ homéomorphes via un homéomorphisme $\varphi$. Il existe alors un voisinage $\U_g$ (resp. $\U_d$) de $N_g$ et (resp. $N_d$) dans $M^{\dagger}$ homéomorphes via un homéomorphisme $\overline{\varphi}$ qui prolonge $\varphi$. Ces voisinages tubulaires permettent de recoller les bords $N_g$ et $N_d$ de $M^{\dagger}$. De façon précise, il existe alors une unique variété $M$ qui possède une sous-variété plongé $N$ homéomorphe à $N_g$ tel qu'il existe une identification de $M^{\dagger}$ avec la variété à bord $M\mid_{N}$ obtenue en découpant $M$ le long $N$. Enfin, cette identification permet d'écrire l'homéomorphisme $\overline{\varphi}$ entre les voisinages tubulaires $\U_g$ et $\U_d$ comme une réflexion dans les cartes.
}
\\
\par{
On peut faire la même chose avec une variété projective $\Mc$ à bord totalement géodésique. Cette fois-ci, il faut prendre un homéomorphisme $\overline{\varphi}$ entre $\U_g$ et $\U_d$ qui lu dans les cartes est une application projective. On obtient ainsi une structure projective sur la variété $M$.
}
\\
\par{
Plier la structure d'une variété hyperbolique $\M$ le long d'une hypersurface propre totalement géodésique $\Nn$ se fait en plusieurs étapes. Tout d'abord, on découpe $\M$ le long de $\Nn$, et on obtient la variété projective à bord totalement géodésique $\Mc$. Ensuite, on remarque que $\M$ est obtenue par le recollage de la variété projective $\Mc$ via un isomorphisme projectif $\overline{\varphi_0}$ entre deux voisinages tubulaires $\U_g$ et $\U_d$. Enfin, on recolle la variété projective $\Mc$ via un isomorphisme projectif $\overline{\varphi_t} = \alpha_t \overline{\varphi_0}$, où $\alpha_t$ est une application projective qui est l'identité sur $\Nn$ et lu dans les cartes est conjugué à la matrice $a_t$.
}

\subsection{Vision géométrique d'un pliage}\label{dessin_pliage}

\begin{center}
\begin{figure}[h!]
  \centering
  \subfloat[][]{\label{pliage1}\includegraphics[trim=0cm 10cm 5cm 0cm, clip=true, width=7cm]{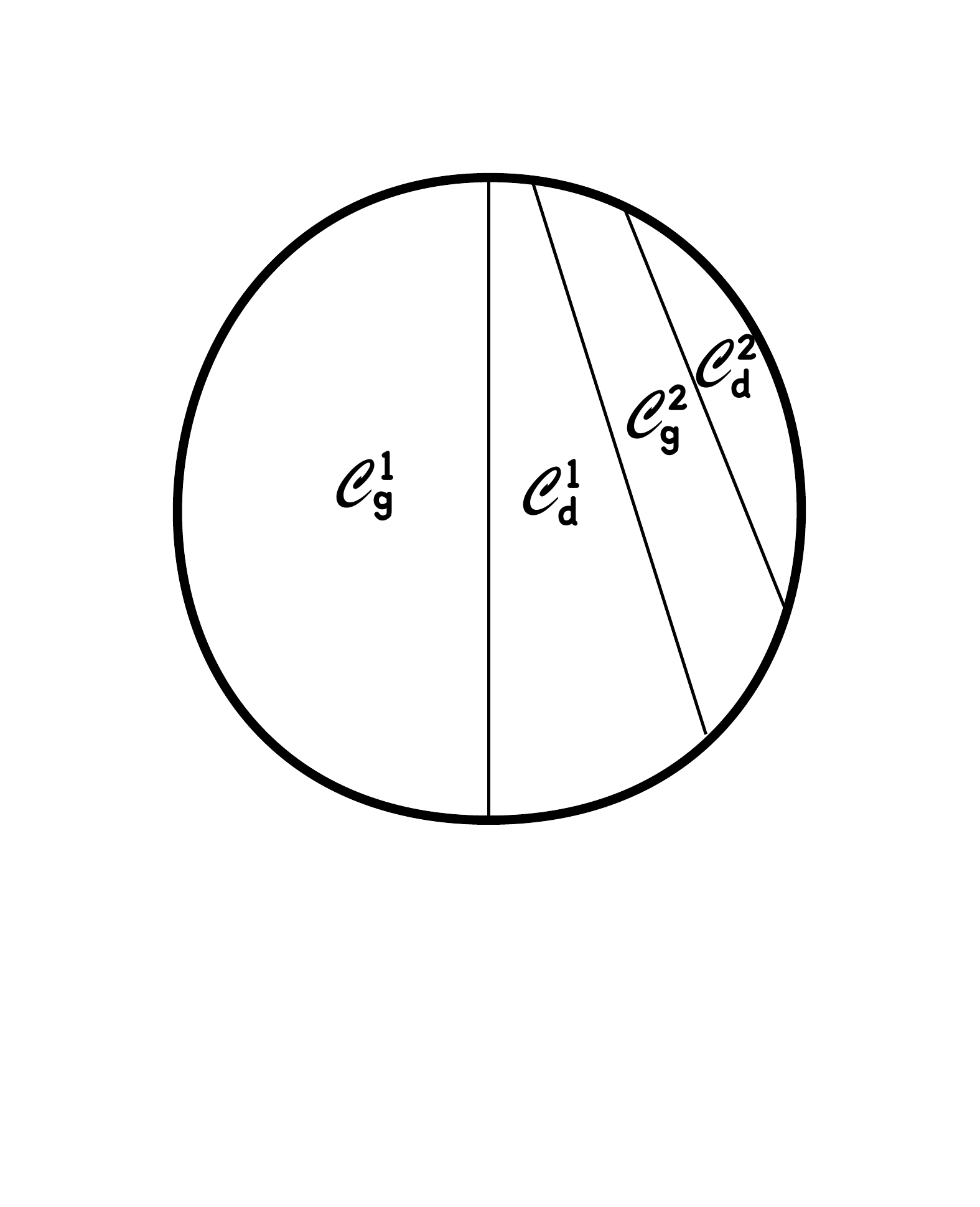}}                
  \subfloat[][]{\label{pliage2} \includegraphics[trim=0cm 10cm 5cm 0cm, clip=true, width=7cm]{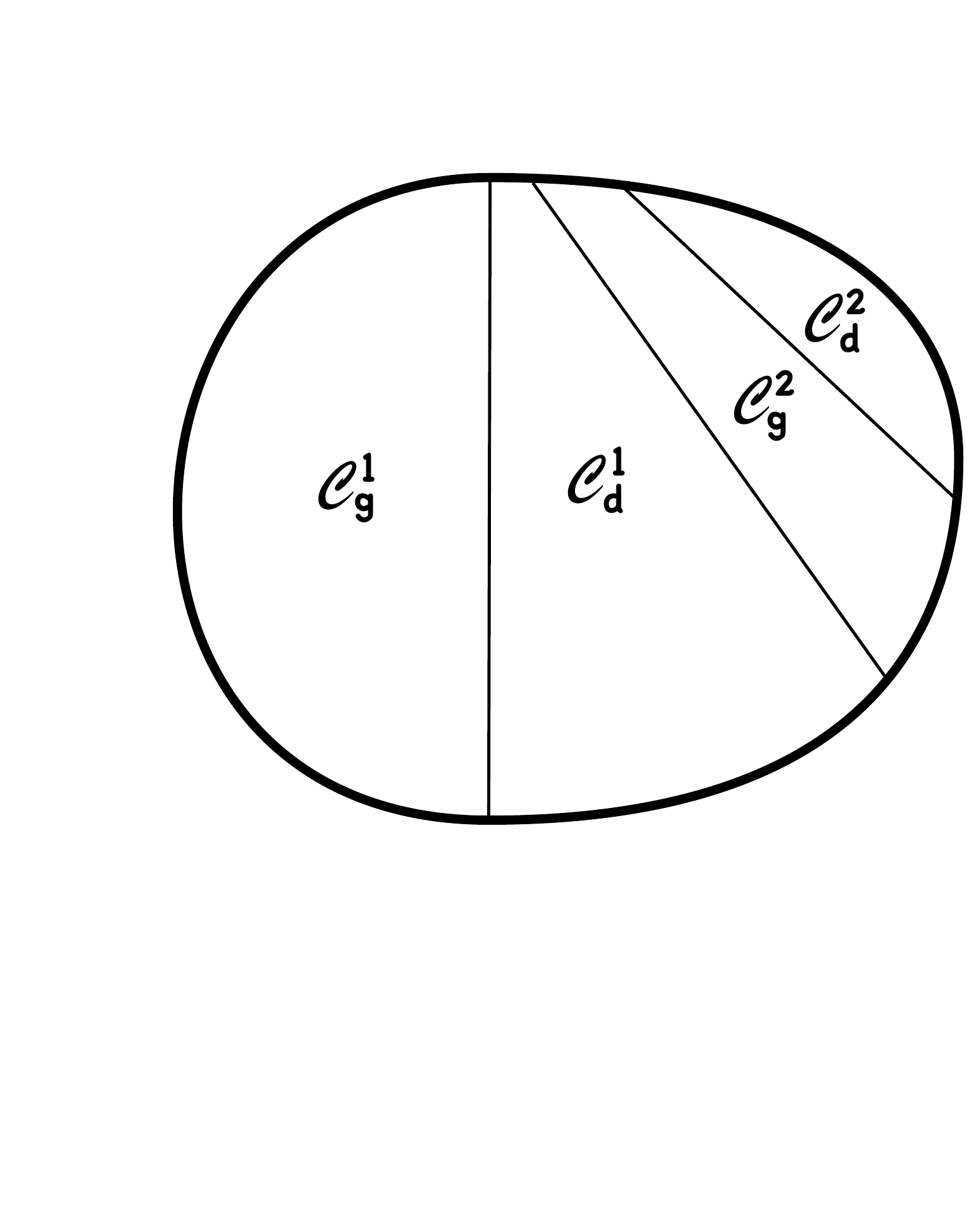}}
\\
  \subfloat[][]{\label{pliage3}\includegraphics[trim=0cm 10cm 5cm 0cm, clip=true, width=7cm]{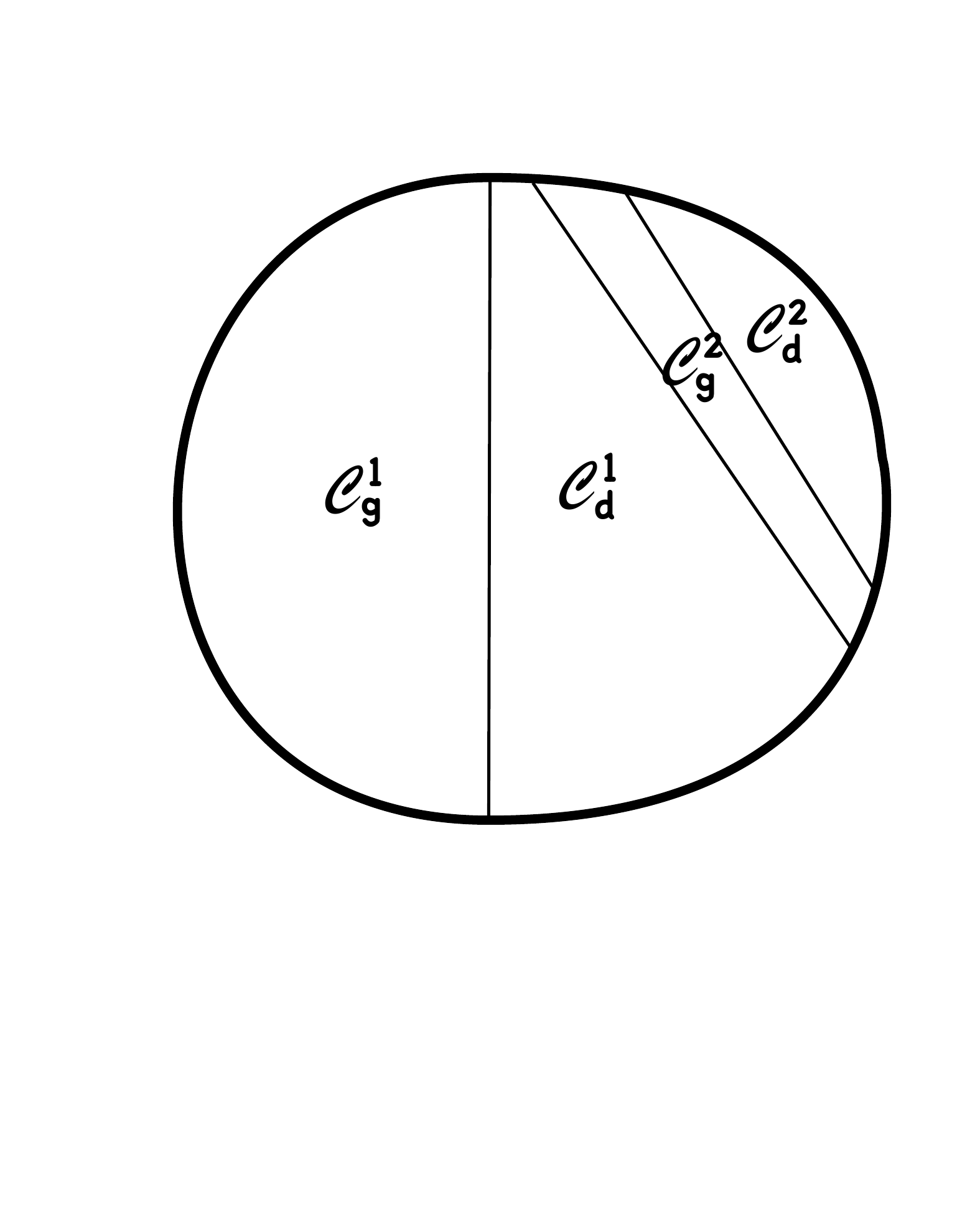}}                
  \subfloat[][]{\label{pliage4} \includegraphics[trim=0cm 10cm 5cm 0cm, clip=true, width=7cm]{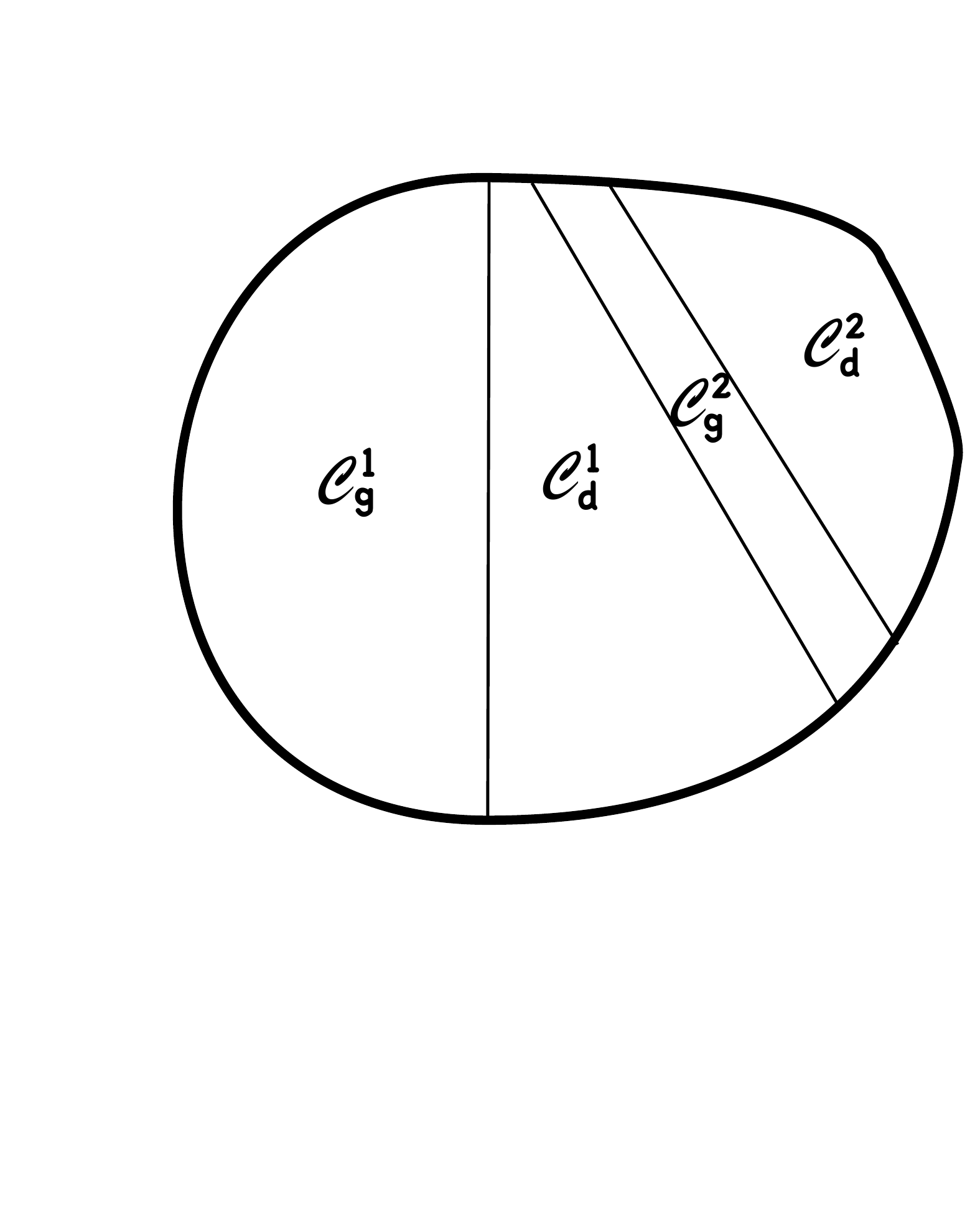}}
\caption{Pliage:\\ Dans la figure (A), le convexe est un ellipsoïde, nous allons le déformer.\\
Dans la figure (B), l'ellipsoïde a subi une transformation, la partie droite a été "gonflé" par une application $A_{H,p,t}$.\\
Dans la figure (C), le convexe de la figure (B) a subi une transformation, la partie à droite de $\C^1_d$ a été "dégonflé" par une application $A_{H,p,-t}$.\\
Dans la figure (D), le convexe de la figure (C) a subi une transformation, la partie à droite de $\C^2_g$  a été "gonflé" par une application $A_{H,p,t}$.
\label{pliage}
}
\end{figure}
\end{center}

\par{
Soient $H$ un hyperplan projectif, $p$ un point de $\PP^n$ qui n'est dans $H$ et $t$ un réel. On définit alors la transformation projective $A_{H,p,t}$ de la façon suivante:
\begin{enumerate}
\item $A_{H,p,t} \in \ss$
\item $A_{H,p,t}$ est l'identité sur $H$
\item $A_{H,p,t}$ fixe le point $p$ et la valeur propre associée à la droite $p$ de $\R^{n+1}$ est $e^{nt}$.
\end{enumerate}
}
\par{
Dans une base convenable, la matrice de $A_{H,p,t}$ est la matrice $a_t$.
}
\\
\par{
Soient  $\O$  un ouvert proprement convexe et $H$ un hyperplan de $\PP^n$ qui rencontre $\O$ et $p$ un point à l'extérieur de $\O$ et de $H$. Une des étapes du pliage d'un ouvert proprement convexe $\O$ revient à appliquer $A_{H,p,t}$ sur une composante connexe de $\O \setminus H$ et l'identité sur l'autre.
}
\\
\par{
Soit $q$ un point dans l'une des deux composantes connexes de $\O \setminus H$. On note $\O_q$ l'adhérence de la composante connexe de $\O \setminus H$ contenant $q$ et $\O_{\overline{q} }$ l'autre adhérence. Enfin, on note $\textrm{Pli}_{\O,H,p,t,q}(\O)$ l'ensemble $\O_q \cup A_{H,p,t}(\O_{\overline{q} } )$.
}
\\
\par{
Il ne semble pas évident à priori que toutes ces transformations vont préserver la convexité de l'ouvert $\O$ et l'injectivité de la développante. Mais c'est le cas. C'est l'objet du théorème \ref{theo_kapo} démontré dans la partie \ref{sec_conv}. Pour comprendre les rouages de cette démonstration, il faut remarquer les 3 points suivants. 
}
\\
\par{
On identifie le revêtement universel de $\widetilde{\M}$ de la variété hyperbolique $\M$ que l'on veut plier le long de l'hypersurface $\Nn$ avec un ellipsoïde $\O$. Les relévés de $\Nn$ à $\O$ définissent des hyperplans projectifs $H_i$ et le dual de l'hyperplan $H_i$ pour la forme quadratique définissant l'ellipsoïde $\O$ est un point $p_i$ de $\PP^n$. Le point $p_i$ est l'intersection des hyperplans tangents à $\partial \O$ en un point de $H_i \cap \partial \O$. Nous appelerons une composante connexe de la préimage de $\Nn$ dans $\O$ un \emph{mur}, et nous appelerons \emph{chambres de $\O$} les adhérences des composantes connexes de $\O$ privé des murs.
}
\\
\par{
Plier la structure projective de $\M$ signifie modifier successivement $\O$ en lui appliquant \underline{des conjuguées} des transformations $A_{H_i,p_i,t}$ ou $A_{H_i,p_i,t}^{-1}$ comme expliquer précédemment (i.e en appliquant des conjuguées de $\textrm{Pli}_{H_i,p_i,t,q_i}$). Ainsi après un nombre dénombrable de transformations (suggéré par la figure \ref{pliage}) on obtient une partie $\O_t$ de l'espace projectif qui est préservé par $\rho_t$. L'un des buts du théorème \ref{theo_kapo} est de montrer que cette partie est en fait un ouvert convexe. La \underline{remarque n°1} est que l'image d'une chambre de $\O$ par toutes ces transformations est encore convexe, puisque $\dev_t$ restreinte à une chambre est une application projective.
}
\\
\par{
La \underline{remarque n°2} est que l'image de la réunion de deux chambres adjacentes de $\O$ par toutes ces transformations est encore convexe. C'est la conséquence du lemme ci-après \ref{lemm_conv_local}. Cette remarque nous assure que tout point possède un voisinage convexe. 
}
\\
\par{
La \underline{remarque n°3} est que si $x$ est un point du bord (dans $\PP^n$) de l'un des $H_i \cap \O$, alors l'image de la réunion des chambres de $\O$ qui contiennent $x$ dans leur adhérence est encore convexe. On a une sorte de "convexité à l'infini". Il y a deux types de points $x$ sur le bord de $H_i \cap \O$, il y a ceux qui correspondent à un cusp de $\Nn$ et qui sont fixés par un groupe parabolique et il y a les autres. Si $x$ est dans la deuxième catégorie alors il y a seulement deux chambres de $\O$ qui le contiennent dans leur adhérence, on est donc ramené à la remarque n°2. Par contre, si $x$ correspond à un cusp de $\Nn$ alors il est inclus dans une infinité de $H_j$ pour $j \in J$. Mais ces $(H_j)_{j \in J}$ ne se rencontrent pas dans $\O$ par conséquent, leur intersection est un sous-espace  projectif de dimension $n-2$, et on peut les énumérer avec $\Z$. De plus, la tangente en $x$ à $\partial \O$ est préservé par tous les $A_{H_j,p_j,t}$. Par conséquent, on obtient le résultat annoncé en appliquant le lemme \ref{lemm_conv_local} à l'aide d'une récurrence.
}

\begin{lemm}\label{lemm_conv_local}
Soient $\O$ un ouvert proprement convexe de $\PP^n$, $H$ un hyperplan de $\PP^n$ qui rencontre $\O$, $p$ un point de $\PP^n$ à l'extérieur de $H$ et de $\O$. Soit $A$ une carte affine contenant $\overline{\O}$ et $p$. On note $\C_1$ et $\C_2$ les deux adhérences des composantes connexes de $\O \setminus H$. Supposons que le demi-cône $\C$ de sommet $p$ et base $H \cap \O$ contienne $\O$ alors l'ensemble $\C_1 \cup A_{H,p,t}(\C_2)$ est convexe.
\end{lemm}

\begin{proof}
Soient $x$ un point de $\C_1$ et $y$ un point de $A_{H,p,t}(\C_2)$. Comme le demi-cône est préservé par $A_{H,p,t}$ et contient $\O$, le point $y$ appartient à $\C$. Comme le demi-cône $\C$ est convexe, il contient le segment $[x,y]$. Ce segment traverse le mur $H \cap \O$ en un point $z$, le segment $[x,z]$ est inclus dans le convexe $\C_1$ et le segment $[z,y]$  est inclus dans le convexe $A_{H,p,t}(\C_2)$.
\end{proof}

\subsection{Un théorème de convexité}\label{theo_conv}

La première difficulté est de comprendre pourquoi cette déformation donne une structure projective PROPREMENT CONVEXE. Signalons tout de même que si la variété $M$ était compacte alors le théorème suivant de Koszul (\cite{MR0239529}) nous assurerait que pour $t$ assez petit la déformation de $\rho_0$ fournit par l'élément $a_t$ serait encore convexe.

\begin{theo}\label{theo_kos}
Soit $M$ une variété compacte, l'espace $\beta(M)$ est ouvert dans $\PP(M)$.
\end{theo}

Nous allons montré le théorème suivant:

\begin{theo}\label{theo_kapo}
Soit $\M$ une variété hyperbolique et $\Nn$ une sous-variété totalement géodésique de $\M$.  Les structures projectives associées au pliage de $\M$ le long de $\Nn$ sont proprement convexe. De plus, notons $\dev_t :\O_0 \rightarrow \O_t$ la développante de la nouvelle structure projective, alors l'application $\dev_t$ se prolonge de façon unique en un homéomorphisme  $\pi_1(\M)$-équivariant: $\dev_t : \overline{\O_0} \rightarrow \overline{\O_t}$ qui induit un homéomorphisme $\dev_t : \partial \O_0 \rightarrow \partial \O_t$.
\end{theo}

La partie \ref{sec_conv} est consacré à la démonstration de ce théorème.

\begin{rema}
Dans \cite{MR2350468}, Misha Kapovich montre une version proche de ce théorème. Au lieu de recoller des variétés projectives convexes le long d'une hypersurface totalement géodésique. Il recolle aussi des variétés projectives convexes \underline{compactes} à coins qui vérifient certaines conditions de compatibilité.
\end{rema}

\begin{nota}
Soit $t >0$, on notera $\Lambda_t$ (resp. $\mho_t$ ) le groupe (resp. l'ouvert proprement convexe) obtenu par le pliage de $\M_0$ le long de $\Nn_0$ à l'aide l'élément $a_t$. On notera $\dev_t: \mho_0 \rightarrow \mho_t$ la nouvelle développante.
\end{nota}

\section{Démonstration du théorème de convexité}\label{sec_conv}

Le but de cette partie est de donner une démonstration du théorème \ref{theo_conv} dont le théorème \ref{theo_kapo} est un corollaire.

\subsection{Définition}
$\,$
\par{
Tout d'abord pour démontrer un résultat de convexité le cadre de $\PP^n$ n'est pas le plus approprié. Nous allons donc nous placer sur la sphère projective $\S$ qui est le revêtement à deux feuillets de $\PP^n$ ou encore l'espace des demi-droites vectorielles de $\R^{n+1}$. On notera $\pi$ la fibration naturelle $\pi: \R^{n+1} \setminus \{ 0 \} \rightarrow \S$
}

\begin{defi}
Une partie $\O$ de $\S$ est dite \emph{convexe} lorsque la réunion $\pi^{-1}(\O) \cup \{0\}$ est une partie convexe de $\R^{n+1}$. Un ouvert convexe $\O$ de $\S$ est dit \emph{proprement convexe} lorsque son adhérence est incluse dans une carte affine de $\S$, ce qui est équivalent au fait que son adhérence ne contient pas de points diamétralement opposés.
\end{defi}

\begin{rema}\label{struct_conv}
Tout ouvert convexe de $\S$ est ou bien $\S$ tout entier ou bien inclus dans une carte affine. De plus, soit $\O$ un ouvert convexe de $\S$, on remarquera que si $E_1$ et $E_2$ sont deux sous-espaces affines inclus dans $\O$ alors il existe un sous-espace affine $E_3$ inclus dans $\O$ dont la direction $\overrightarrow{E_3}$ est $\overrightarrow{E_1} \oplus \overrightarrow{E_2}$. Tout ouvert convexe de $\S$ possède donc une direction maximale $\overrightarrow{E_{\O}}$ égale au sous-espace vectoriel engendré par l'intersection $\O \cap - \O$ dans $\R^{n+1}$. Enfin, la projection de $\O$ dans la sphère projective quotient $\mathbb{S}\big(\R^{n+1}/_{\overrightarrow{E_{\O}}}\big)$ est un ouvert proprement convexe.
\end{rema}

\par{
Nous allons avoir besoin d'un peu de vocabulaire. Soit $\M$ une variété projective. On peut définir la notion de segment et de convexité sur le revêtement universel $\widetilde{\M}$ de $\M$.
}

\begin{defi}
Un \emph{segment} de $\widetilde{\M}$ est une application $s:[0,1] \rightarrow \widetilde{\M}$ tel que la compos\'ee $\dev \circ s : [0,1] \rightarrow \S$ est une application continue injective qui définit un segment de $\S$ de longueur inférieure ou égale à $\pi$ pour la distance canonique sur $\S$.

Une partie $A$ de $\widetilde{\M}$ est dite \emph{convexe} lorsque tout couple de points de $A$ peut être joint par un segment.
\end{defi}

%
%

\begin{defi}
Soit $\M$ une variété projective. On se donne $(W_i)_{i \in I}$ une famille localement finie d'hypersurfaces propres totalement géodésiques de $\M$. On appelera les  $(W_i)_{i \in I}$ des \emph{murs}. On appelera \emph{chambre} l'adhérence de toutes composantes connexes de $\displaystyle{\M \setminus \bigcup_{i \in I} W_i }$. On dira que deux chambres sont \emph{adjacentes} lorsque leur intersection est incluse dans un unique mur.

L'intersection de deux murs $W_1$ et $W_2$ est vide ou une sous-variété propre totalement géodésique de codimension 2. Lorsqu'elle est non vide, on dira que l'intersection de deux murs $W_1$ et $W_2$ est \emph{incluse dans une situation diédrale} lorsqu'il existe un entier $m \geqslant 2$, une suite de $m$ murs $(W_i)_{i=1...m}$, et une suite de $2m$ chambres $(C_i)_{i=1...2m}$ tel que (Voir figure \ref{diedral}):
\begin{itemize}
\item l'intersection de ces $m$ murs et $2m$ chambres soient l'intersection $W_1 \cap W_2$.
\item Deux chambres consécutives sont adjacentes.
\item Le mur contenant l'intersection des deux chambres consécutives $\C_i$ et $\C_{i+1}$ est aussi le mur contenant l'intersection des deux chambres consécutives $\C_{i+m}$ et $\C_{i+1+m}$.
\end{itemize}
\end{defi}

\begin{figure}[!h]
\begin{center}
\includegraphics[trim=-4cm 18cm 4cm 1cm, clip=true, width=14cm]{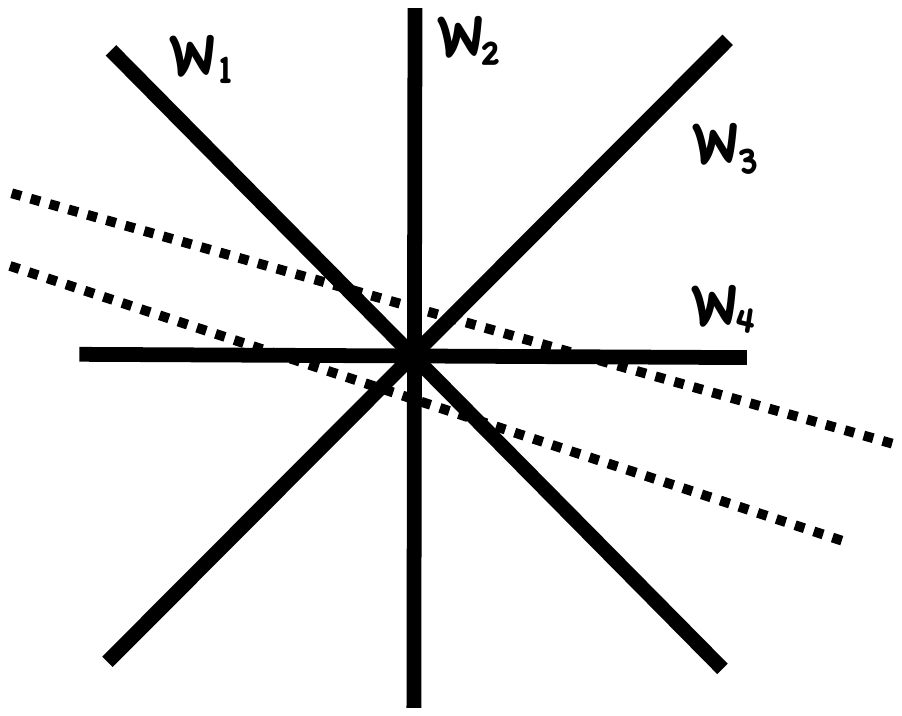}
\caption{Situation diédrale} \label{diedral}
\end{center}
\end{figure}

\begin{rema}
L'image d'un mur par la développante de la structure projective de $\M$ est un ouvert convexe d'un hyperplan projectif de $\S$.
\end{rema}

\subsection{Le th\'eor\`eme}

\begin{rema}\label{prolongement}
L'application $\dev$ est un hom\'eomorphisme local et l'espace $(\PP^n,d_{can})$ est un espace m\'etrique, il existe donc une unique distance sur $\widetilde{\M}$ tel que l'application $\dev$ est une isométrie locale. On note $Cl(\widetilde{\M})$ le complété de $\widetilde{\M}$ pour cette distance. L'espace $Cl(\widetilde{\M})$ est compact et l'application $\dev$ se prolonge en une application continue encore notée $\dev$ de $Cl(\widetilde{\M})$ vers $\S$.
\end{rema}

\begin{theo}\label{theo_conv}
Soit $\M$ une variété projective. On se donne une famille finie d'hypersurfaces propres totalement géodésiques $(H_i)_{i \in I}$. Cette famille définit une famille localement finie de murs du revêtement universel $\widetilde{\M}$. On suppose que:
\begin{enumerate}
\item Les chambres de $\widetilde{\M}$ sont convexes.
\item La réunion de deux chambres adjacentes est convexe.
\item Toute intersection non vide de deux murs est incluse dans une situation diédrale.
\item Pour tout mur $W$, et tout point $x_{\infty} \in \partial W  \subset Cl(\widetilde{\M})$, la réunion des chambres contenant  $x_{\infty}$ dans leur adhérence (dans $Cl(\widetilde{\M})$) est convexe.
\end{enumerate}
Alors, la variété projective $\M$ est convexe.
\end{theo}

\subsection{Démonstration du théorème \ref{theo_conv}}
$\,$
\par{
La démonstration de ce théorème se déroule en plusieurs étapes. La difficulté est d'obtenir le lemme suivant:
}
\begin{lemm}\label{lemm_conv}
L'ensemble $\widetilde{\M}$ est convexe.
\end{lemm}

Nous allons commençer par montrer comment ce lemme entraine le th\'eor\`eme \ref{theo_conv}. Ensuite nous montrerons ce lemme à l'aide d'un argument de connexité.

\begin{proof}[D\'emonstration du th\'eor\`eme \ref{theo_conv} \`a l'aide du lemme \ref{lemm_conv}]
On doit montrer que l'application  $\dev$ est un homéomorphisme sur son image $\O$ et que $\O$ est une partie convexe de $\PP^n$.

L'application $\dev$ est un homéomorphisme local pour montrer que c'est un homéomorphisme sur son image, il suffit donc de montrer qu'elle est injective. 

Soient $x$ et $y$ deux points de $\widetilde{\M}$, il existe un segment $s$ qui relie $x$ \`a $y$ dans $\widetilde{\M}$, l'application $\dev \circ s$ est injective donc si $\dev(x) =\dev(y)$ alors $x=y$. L'application $\dev$ est donc un homéomorphisme sur son image $\Omega$.

Enfin, l'ouvert $\Omega$ est convexe puisque c'est l'image de $\widetilde{\M}$ qui est convexe.
\end{proof}

\subsection{D\'emonstration du lemme \ref{lemm_conv}}
$\,$
\par{
On appelle \emph{point singulier} de $\widetilde{\M}$ tout point inclus dans l'intersection de trois murs $W_1$, $W_2$ et $W_3$ tel que $W_1 \cap W_2 \cap W_3$ est une sous-variété totalement géodésique de codimension 3. On note $Sing$ l'ensemble des points singuliers de $\widetilde{\M}$.
}
Soient $x$ un point de $\widetilde{\M}$ et $\mathcal{C}$ une chambre de $\widetilde{\M}$.
On d\'efinit les ensembles suivants:

$$\C^{reg} = \{ y \in \C \textrm{ tel qu'il existe un segment dans }  \widetilde{\M}- Sing \textrm{ reliant } x \textrm{ \`a } y \}$$

$$\C^* = \{ y \in \C \textrm{ tel qu'il existe un segment dans }  \widetilde{\M} \textrm{ reliant } x \textrm{ \`a } y \}$$

\begin{lemm}\label{lemm_const}
L'application qui a $y \in \C^{reg}$ associe le nombre et l'ensemble des murs traversées par le segment $[x,y]$ est localement constante.
\end{lemm}

\begin{proof}
C'est l'hypothèse "toute intersection non vide de deux murs est incluse dans une situation diédrale" qui donne ce lemme. Il suffit de regarder la figure \ref{diedral}.
\end{proof}

\begin{lemm}\label{lemm_ouvert}
L'ensemble $\C^{reg}$ est ouvert dans $\C$ .
\end{lemm}

\begin{lemm}\label{lemm_ferme}
Les composantes connexes de $\C^*$ sont fermées dans $\C$.
\end{lemm}

\begin{coro}\label{coro_conv}
Si $\C^{reg}$ est non vide alors $\C^* = \C$.
\end{coro}

Commençons par montrer que ces lemmes entrainent le corollaire \ref{coro_conv} et le lemme \ref{lemm_conv}.

\begin{proof}[Démonstration du corollaire \ref{coro_conv} à l'aide des lemmes \ref{lemm_ouvert} et \ref{lemm_ferme}]
Nous allons montrer que $\C^*$ est dense dans $\C$. Soit $y \in \C^{reg}$, on note $A^{reg}$ une composante connexe de $\C^{reg}$. Le nombre de murs traversés par les segments $[x,y]$ pour $y \in A^{reg}$ ne dépend pas de $y$ (lemme \ref{lemm_const}). On note $A^*$ la composante connexe de $\C^*$ contenant $A^{reg}$. Par hypothèse, $A^*$ est fermé dans $\C$.

Les murs forment une famille localement finie dans $\widetilde{\M}$, par conséquent $Sing$ est une réunion localement finie de sous-variétés totalement géodésique de codimension 3.

Comme \underline{l'ensemble} des murs traversées par tout segment $[x,y]$ est constant pour $y \in A^{reg}$, il existe un fermé $S$ de codimension 2 de $\C$ tel que $A^{reg} \cap S^c = A^* \cap S^c$ ($S^c = \C \setminus S$).  Ainsi, les lemmes \ref{lemm_ouvert} et \ref{lemm_ferme} montrent que $A^{reg} \cap S^c = A^* \cap S^c$ est ouvert et fermé dans $S^c$. Or, $S$ est de codimension 2 par conséquent $S^c$ est connexe puisque $\C$ est convexe.

Ensuite, $S^c$ est dense dans $\C$ toujours car $S$ est de codimension 2, par conséquent $\C^{reg}$ est non vide si et seulement si $\C^{reg} \cap S^c$ est non vide puisque $\C^{reg}$ est ouvert.

Il vient donc que $A^{reg} \cap S^c = A^* \cap S^c = \C \cap S^c$. Le lemme \ref{lemm_ferme} montre que $A^* = \C$, puisque $S^c$ est dense dans $\C$. Il vient que si $\C^{reg}$ est non vide alors $\C^* = \C$.
\end{proof}

\begin{proof}[Démonstration du lemme \ref{lemm_conv} à l'aide du corollaire \ref{coro_conv}]
A présent, pour montrer que $\widetilde{\M}$ est convexe, il suffit de choisir un point $x$ de $\widetilde{\M}$.  On dira que deux points $x$ et $y$ sont à distance combinatoire inférieure ou égale à $n$ si et seulement s'il existe une suite de $n$ chambres $C_1,...,C_n$ tel que $x \in C_1$, $y \in C_n$, $C_i$ et $C_{i+1}$ sont adjacentes.  

Le corollaire \ref{coro_conv} montre facilement à l'aide d'une récurrence que la réunion des chambres à distance combinatoire inférieure ou égale à $n$ de $x$ est étoilé par rapport à $x$. L'ensemble $\widetilde{\M}$ est donc convexe.
\end{proof}

\subsection{Démonstration du lemme \ref{lemm_ouvert}}


\begin{proof}[Démonstration du lemme \ref{lemm_ouvert}]
Soit $y \in \C^{reg}$, il existe un segment $s$ reliant $x$ \`a $y$ dans $\widetilde{\M} \setminus Sing$. Comme l'image d'un segment est compact elle est incluse dans un nombre fini de chambres de $\widetilde{\M}$. On note $(x^i)_{i=1...N}$ les points d'intersections du segment $[x,y]$ avec les murs de $\widetilde{\M}$ num\'erot\'es via la param\'etrisation de $[x,y]$. On pose $x^0=x$ et $x^{N+1}=y$.

Comme les chambres sont convexes et $\widetilde{\M} \setminus Sing$ est ouvert, il existe des voisinages $V_{x^i}$ de $x^i$ (pour $i=1,...,N+1$) dans  $\widetilde{\M}$ tels que l'enveloppe convexe dans chaque chambre des couples $(V_{x^i},V_{x^{i+1}})$ contienne un voisinage convexe du segment $[x^i,x^{i+1}]$ inclus dans la chambre contenant $[x^i,x^{i+1}]$ et qui ne rencontre pas $Sing$. La réunion de ces voisinages contient un voisinage convexe de $[x,y]$. Il existe donc un voisinage ouvert du point $y$ dans $\C^{reg}$.
\end{proof}

\subsection{D\'emonstration du lemme \ref{lemm_ferme}}

%
%
$\,$
\par{
L'énoncé du lemme suivant est assez technique.
}
\begin{lemm}\label{bord}
Soit $(s_n)_{n \in \N}$ une suite de segments d'extrémités le point $x$ et un point  $y_n$ appartenant à une composante connexe de  $\C^*$ fixée. Le segment $s_n$ traverse $N_n$ murs $W_1^n,...,W^n_{N_n}$ de $\widetilde{\M}$ ( ordonnés par la paramétrisation de $s_n$). On note $(x_n^{i})_{i=1...N_n}$ les points d'intersection de $s_n$ avec les murs de  $\widetilde{\M}$ ( ordonnés via la paramétrisation de $s_n$). Si la suite $(x_n^1)_{n \in \N}$ diverge dans $\widetilde{\M}$ et converge dans $Cl(\widetilde{\M})$ vers $x^1_{\infty}$, alors, si $n$ est assez grand, la suite $N_n$ est constante égale à un certain $N$, les suites $W^n_1,...,W^n_{N}$ sont constantes et les suites $(x_n^i)_{n \in \N}$ pour $i=2...N$ diverge dans $\widetilde{\M}$ et converge dans $Cl(\widetilde{\M})$ vers $x^1_{\infty}$.
\end{lemm}

\begin{proof}[D\'emonstration du lemme \ref{bord}]
Comme toutes les intersections de murs sont incluses dans une situation diédrale, le lemme \ref{lemm_const} montre que l'ensemble des murs traversés est constant et donc fini. Par conséquent, on peut supposer que les segments $s_{n}$ traversent $N$ murs $\widetilde{\M}$ et que la suite $W_1^n,...,W^n_{N}$ est constante. La suite de segments $[x,x_{n}^2]$ est incluse dans la réunion de deux chambres adjacentes : $\C_1$ et $\C'$ adjacentes au mur $W_{1}$. La réunion de ces deux chambres est convexe et la suite $x_{n}^1$ converge vers le point $x_{n}^1$ de $Cl(\widetilde{\M})$. Par suite, la suite $x_{n}^2$ diverge dans $\widetilde{\M}$ et converge dans $Cl(\widetilde{\M})$ vers $x^1_{\infty}$.

On itère ce raisonnement pour obtenir la conclusion pour toutes les suites $(x_n^i)_{n \in \N}$ pour $i=2...N$.
\end{proof}

\begin{proof}[D\'emonstration du lemme \ref{lemm_ferme}]
Soit $y_{n}$ une suite de points d'une composante connexe de $\C^*$. Supposons que la suite $(y_n)_{n \in \N}$ converge dans $\widetilde{\M}$ vers un point $y_{\infty}$. Nous allons montrer qu'il existe un segment entre $x$ et $y_{\infty}$. Il existe un segment $s_{n}$ reliant $x$ \`a $y_n$ dans $\widetilde{\M}$. On vient de voir (lemme \ref{bord}) que les segments $s_{n}$ traversent $N$ murs.
On note $(x_n^i)_{i=1...N}$ les points d'intersections du segment $s_{n}$ avec les murs de $\widetilde{\M}$ num\'erot\'es via la param\'etrisation de $s_{n}$. On pose $x_n^0=x$ et $x_n^{N+1}=y_n$.

On a trois cas \`a distinguer:
\begin{itemize}
\item Toutes les suites $(x_n^{i})_{n \in \N}$ divergent. En particulier, la suite $(x_n^1)_{n \in \N}$ diverge dans $X$, quitte \`a extraire, on peut supposer qu'elle converge dans $Cl(\widetilde{\M})$.

\item Il existe un $i_0=2,...,N$ tel que la suite $(x_n^{i_0})_{n \in \N}$ diverge dans $\widetilde{\M}$ mais sous-converge dans $Cl(\widetilde{\M})$ et la suite $(x_n^{i_0-1})_{n \in \N}$ converge dans $\widetilde{\M}$.

\item Toutes les suites $(x_n^{i})_{n \in \N}$ convergent dans $\widetilde{\M}$.
\end{itemize}
Nous allons montrer que les deux premiers cas sont absurdes.
Le lemme \ref{bord} montre que dans le premier cas les suites $(x_n^i)_{n \in \N}$ pour $i=1...N$ converge vers un point $x_{\infty}^1$ de $Cl(\widetilde{\M}) - \widetilde{\M}$.

La quatrième hypothèse du théorème \ref{theo_conv} affirme que la réunion des chambres contenant le point $x_{\infty}^1$ dans leur adhérence dans $Cl(\widetilde{\M})$ est convexe.

Par cons\'equent la suite $y_{n}$ converge $x^1_{\infty}$ qui n'est pas dans $\widetilde{\M}$.  Ce qui est absurde. L'absurdit\'e du second cas se d\'emontre exactement de la m\^eme mani\`ere.

Par suite, toutes les suites $(x_n^{i})_{n \in \N}$ convergent dans $\widetilde{\M}$. Il vient que les points $x$ et $y_{\infty}$ sont reli\'es par une r\'eunion finie de segments qui v\'erifie de plus que sa restriction \`a la r\'eunion de deux chambres adjacentes est un segment. Ce chemin est donc un segment.
\end{proof}

\subsubsection{Démonstration du théorème \ref{theo_kapo}}

Pour terminer la démonstration du théorème \ref{theo_kapo}, il faut montrer que dans le cas qui nous intéresse la structure projective convexe est proprement convexe. Ce n'est pas très difficile. C'est une conséquence de la proposition suivante et du fait que les représentations $\rho_t$ sont irréductibles.

\begin{prop}
L'holonomie d'une structure projective convexe non proprement convexe n'est pas irréductible.
\end{prop}

\begin{proof}
La remarque \ref{struct_conv} montre que l'holonomie d'une structure projective convexe doit préserver l'espace vectoriel engendré par l'intersection $\O \cap -\O$. Une structure projective convexe est non proprement convexe si et seulement si cette intersection est non vide.
\end{proof}

\begin{proof}[Démonstration du théorème \ref{theo_kapo}]
Il nous reste à montrer que l'application  $\dev_t :\O_0 \rightarrow \O_t$ se prolonge de façon unique en un homéomorphisme  $\pi_1(\M)$-équivariant: $\dev_t : \overline{\O_0} \rightarrow \overline{\O_t}$ qui induit un homéomorphisme $\dev_t : \partial \O_0 \rightarrow \partial \O_t$.

Pour cela, on identifie $\widetilde{\M}$ avec $\O_0$, et il faut montrer que l'on peut identifier $Cl(\widetilde{\M})$ et $\overline{\O_0}$. L'espace $Cl(\widetilde{\M})$ est le complété de $\widetilde{\M}$ muni de la distance qui fait de l'application $\dev_t :\widetilde{\M} \rightarrow \S$ une isométrie locale. Les chambres de $\widetilde{\M}$ forment une partition de $\widetilde{\M}$ en partie convexe. L'application $\dev_t$ restreinte aux chambres de $\widetilde{\M}$ est une application projective. L'adhérence d'une chambre de $\widetilde{\M}$ pour la métrique induite par $\dev_t$ correspond donc à l'adhérence d'une chambre de $\O_0$ dans $\S$.

La remarque \ref{prolongement} montre que l'application $\dev_t$ se prolonge en une application continue de $Cl(\widetilde{\M}) = \overline{\O_0}$ vers $\S$. Cette application restreinte à $\O_0$ est un homéomorphisme sur son image et son image est le convexe $\O_t$ par le théorème \ref{theo_conv}.

Le prolongement de $\dev_t$ est donc un homéomorphisme $\overline{\O_0} \rightarrow \overline{\O_t}$, $\pi_1(\M)$-équivariant qui induit un homéomorphisme $\dev_t : \partial \O_0 \rightarrow \partial \O_t$.
\end{proof}

\section{Irréductibilité, Zariski-Densité et Minimalité} \label{sec_zar}

\subsection{Irr\'eductibilit\'e}

\begin{defi}
Soit $\Gamma$ un sous-groupe de $\ss$. On dira que $\G$ est \emph{irréductible} ou que $\G$ \emph{agit de façon irréductible} sur $\R^{n+1}$ lorsque les seuls sous-espaces vectoriels de $\R^{n+1}$ invariant par $\G$ sont $\{0 \}$ et $\R^{n+1}$.
\end{defi}

\begin{defi}
Soit $\Gamma$ un sous-groupe discret et infini de $\SO$. On se donne $x_{0}$ un point de $\H$ et on d\'efinit l'ensemble:
$$
L_{\Gamma}^{x_{0}}= \{  x_{\infty} \in \partial \H \textrm{ tel qu'il existe une suite d'éléments } \g_{n} \in \Gamma \textrm{ tel que } \g_{n} \cdot x_{0} \underset{n \rightarrow \infty}\rightarrow x_{\infty}  \}
$$

Cet ensemble ne d\'epend pas du point $x_{0}$. On l'appele \emph{l'ensemble limite} de $\Gamma$ et on le note $L_{\Gamma}$. Lorsque $\G$ n'est pas virtuellement abélien c'est le plus petit fermé non vide invariant par $\Gamma$ et c'est l'adhérence des points attractifs des éléments de $\Gamma$. On pourra consulter le livre (\cite{MR2249478}) pour avoir des détails.
\end{defi}

\begin{lemm}\label{lem_irr_SO}
Soit $\Gamma$ un sous-groupe de $\SO$. Si $\Gamma$ n'est pas irréductible alors l'ensemble limite de $\Gamma$ est inclus dans un hyperplan de $\partial \H$.
\end{lemm}

\begin{proof}
Si $\G$ est virtuellement abélien alors l'ensemble limite de $\G$ contient au plus deux points, il est donc inclus dans un hyperplan de $\partial \H$.

Si $\G$ n'est pas virtuellement abélien alors l'ensemble limite de $\G$ est le plus petit fermé non vide $\G$-invariant de $\partial \H$. Nous allons montrer que $\G$ préserve un fermé \underline{non vide} de $\partial \H$ de la forme $F \cap \partial \H$, où $F$ est un sous-espace vectoriel de $\R^{n+1}$.

Comme le groupe $\Gamma$ n'est pas irréductible, il préserve un sous-espace vectoriel $E$ de $\R^{n+1}$ de dimension $p$.  Notons $\C = \{ x \in \R^{n+1} \mid q(x) < 0 \}$ le cône de lumière de $q$.

La forme quadratique $q$ restreinte à $E$ possède trois signatures possibles:

\begin{tabular}{clcc}
$\cdot$ & $(p,0)$      & ou de façon équivalente & $E \cap \overline{\C} = \varnothing$ \\
$\udotdot$& $(p-1,1)$  &               "                       & $E \cap \C \neq \varnothing$\\
$\therefore$ & $(p-1,0)$  &           "                           & $E \cap \C = \varnothing$ et $E \cap \partial \C \neq \varnothing$ \\
\end{tabular}

Dans les deux derniers cas, l'intersection $E \cap \partial \H$ est non vide et préservé par $\Gamma$, dans le premier cas l'intersection de $E^{\bot}  \cap \partial \H$ est non vide et préservé par $\Gamma$. Par conséquent, l'ensemble limite de $\Gamma$ est inclus dans un hyperplan de $\partial \H$.
\end{proof}

\subsection{Ensemble limite}

La définition d'ensemble limite pour un sous-groupe discret $\G$ de $\ss$ n'est pas aussi simple que pour un sous-groupe de $\SO$, même s'il préserve un ouvert proprement convexe. En effet, on ne peut pas définir l'ensemble limite à l'aide d'un point base $x_0$, car c'est la strict-convexité de l'ellipsoïde qui donne l'indépendance en $x_0$ de l'ensemble limite $L_{\G}^{x_0}$.

Mais Yves Benoist a montré le théorème suivant qui nous permet de définir l'ensemble limite dans un cadre plus général (lemme 2.5 de \cite{MR1767272} ou  lemme 3.6 \cite{MR1437472}).

\begin{theo}[Benoist]\label{def_enslim}
Soit $\G$ un sous-groupe discret de $\ss$ qui préserve un ouvert proprement convexe $\O$ de $\PP^n$. Si $\G$ est irréductible alors il existe un unique fermé non vide $\G$-invariant $L_{\G}$ tel que si $F \subset \PP^n$ est un fermé non vide $\G$-invariant alors $L_{\G} \subset F$. On appelle ce fermé \emph{l'ensemble limite} de $\G$, on le note $L_{\G}$, c'est l'adhérence des points attractifs des éléments de $\G$.
\end{theo}

\begin{prop}\label{prop_min}
Le groupe $\Lambda_t$ que l'on a construit est irréductible. Par conséquent son ensemble limite est bien défini. L'action du groupe $\Lambda_t$ sur le bord $\partial \mho_t$ du convexe $\mho_t$ est minimale (i.e il n'existe pas de fermé non triviale invariant par $\G$), autrement dit l'ensemble limite du groupe $\Lambda_t$ est égale à $\partial \mho_t$.
\end{prop}

\begin{proof}
La variété projective $\M_t$ a été construite à l'aide d'un pliage de la variété hyperbolique $\M_0$ le long de l'hypersurface $\Nn_0$. On note $H$ un relevé de $\Nn_0$ à $\mho_t$ pour tout $t\in \R$. 

Les composantes connexes de $\mho'=\displaystyle{\mho_t \setminus \bigcup_{\g \in \Lambda_t} \g H}$ sont des convexes inclus dans $\mho_t$.
Si $\M\setminus\Nn$ possède deux composantes connexe alors on note $\C_t^1$ et $\C_t^2$ les deux composantes connexes de $\mho'$ qui bordent $H$. Si $\M\setminus \Nn$ est connexe alors on note $\C_t^*$ l'une des deux composantes connexes de $\mho'$ qui bordent $H$.

Les arguments qui suivent ne dépendent pas du cas dans lequel on est. On se place donc dans l'un des deux cas, on note $\C_t$ le convexe $\C_t^1$, $\C_t^2$ ou $\C_t^*$.

Le stabilisateur $\Lambda_t^{mor}$ de $\C_t$ est conjugué à un sous-groupe discret de $\SO$ et la variété $\C_t /_{ \Lambda_t^{mor}}$ est une variété hyperbolique non complète. De plus, le volume de tout fermé inclus dans $\C_t/_{ \Lambda_t^{mor}}$ est fini. L'ensemble limite de $\Lambda_t^{mor}$ est donc $\displaystyle{\partial \C_t \setminus \bigcup_{\g \in \Lambda_t} \g H}$.

Par conséquent, l'action de $\Lambda_t^{mor}$ sur $\R^{n+1}$ est irréductible (lemme \ref{lem_irr_SO}). Il vient que le groupe $\Lambda_t$ est aussi irréductible. En particulier, son ensemble limite est bien défini (théorème \ref{def_enslim}).

De plus, l'ensemble limite de $\Lambda_t$ contient les ensembles limites de tous les stabilisateurs des composantes connexes de $\displaystyle{\mho_t \setminus \bigcup_{\g \in \Lambda_t} \g H}$. Par suite, l'ensemble limite de $\Lambda_t$ est $\partial \mho_t$ (i.e l'action de $\Lambda_t$ sur $\partial \mho_t$ est minimale).
\end{proof}

\subsection{Zariski-densité}

Le théorème suivant est dû à Yves Benoist dans \cite{MR1767272}.

\begin{theo}[Benoist]
Soit  $\G$ un sous-groupe discret de $\ss$ qui préserve un ouvert proprement convexe $\O$ de $\PP^n$. Si l'action de $\G$ sur le bord de $\O$ est minimale alors l'adhérence de Zariski de $\G$ est conjuguée à $\SO$ ou $\ss$.
\end{theo}

\begin{coro}\label{coro_zar}
Les groupes $\Lambda_t$ que l'on a construit sont Zariski dense dans $\ss$.
\end{coro}

\begin{proof}
Il ne nous reste plus qu'à montrer que l'adhérence de Zariski $G$ de $\Lambda_t$ ne peut pas être conjuguée à $\SO$.

On reprend les notations de la démonstration de la proposition \ref{prop_min}. Les stabilisateurs des composantes connexes de $\displaystyle{\mho_t \setminus \bigcup_{\g \in \Lambda_t} \g H}$ sont des sous-groupes discrets irréductibles de différents conjugués de $\SO$.

Or, le lemme \ref{lem_zar_so} montre que ces groupes sont Zariski denses dans le conjugué de $\SO$ qui le contient. Par conséquent, le groupe $G$ ne peut être conjugué à $\SO$.
\end{proof}

\begin{lemm}\label{lem_zar_so}
Tout sous-groupe discret et irréductible de $\SO$ est Zariski dense dans $\SO$.
\end{lemm}

On pourra trouver une démonstration de ce lemme dans \cite{MR2081159}.

\begin{rema}
On aurait pu montrer que $\Lambda_t$ est Zariski-dense en utilisant simplement le fait que $\SO$ est un sous-groupe fermé connexe maximale de $\ss$. Nous avons choisi de montrer que l'action de $\Lambda_t$ sur $\mho_t$ est minimale car cela nous sera utile par la suite.
\end{rema}

\subsection{Quelques conséquences}\label{csq}
\subsubsection{Un pliage est une déformation non triviale}

\begin{coro}
Le pliage d'une variété hyperbolique de volume fini le long d'une hypersurface totalement géodésique définit une déformation non triviale de la structure projective.
\end{coro}

\begin{proof}
Supposons que les représentations $\rho_t$ et $\rho_{t'}$ sont conjugués par un élément $g \in \ss$. Il est clair que $\rho_t \neq \rho_{t'}$ si $t \neq t'$. On ne fait que le cas $\M \setminus \Nn$ possède deux composantes connexes. L'autre cas est analogue.

Les représentations $\rho_t$ et $\rho_{t'}$ sont égales sur $\pi_1(\M_g)$. L'adhérence de Zariski de $\rho_t(\pi_1(\M_g))$ est $\SO$ (lemme \ref{lem_zar_so}), par suite, $g$ appartient au centralisateur de $\SO$ dans $\ss$, c'est à dire au centre de $\ss$. Par suite $\rho_t= \rho_{t'}$, ce qui est absurde.



\end{proof}

\subsubsection{Le groupe $\Lambda_t$ n'est pas un réseau de $\ss$}

\begin{rema}
Le groupe $\Lambda_t$ n'est pas un réseau de $\ss$, car l'action d'un réseau de $\ss$ sur l'espace $\ss$-homogène $\PP^n$ est ergodique. C'est une conséquence du théorème d'ergodicité de Moore et du théorème de dualité, on pourra consulter le livre \cite{MR1781937} (notamment l'exemple 2.9 page 92).
\end{rema}

\subsubsection{L'ouvert $\mho_t$ n'est pas homogène}

\begin{prop}
Le groupe $\Aut(\mho_t)$ est discret par conséquent, le groupe $\Lambda_t$ est d'indice fini dans le groupe $\Aut(\mho_t)$.
\end{prop}

\begin{proof}
Cette proposition est une conséquence directe de la proposition \ref{prop_zar} ci-dessous et du corollaire \ref{coro_zar} ci-dessus.
\end{proof}

La proposition suivante est connue depuis longtemps.

\begin{prop}\label{prop_zar}
Tout sous-groupe $\G$ Zariski-dense d'un groupe de Lie quasi-simple $G$ est discret ou dense.
\end{prop}

\begin{proof}
Soit $H$ l'adhérence de $\G$ pour la topologie séparée de $G$, on note $H_0$ la composante neutre de $H$. Le groupe $H_0$ est normalisé par un sous-groupe d'indice fini de $\G$ car $H$ possède un nombre fini de composantes connexes puisque c'est un groupe algébrique. Par suite, $H_0$ est normalisé par $G$ puisque $\G$ est Zariski-dense. Comme $G$ est quasi-simple, $H_0$ est égale à $G$ ou $\{1 \}$. Par suite, $\G$ est discret ou dense.
\end{proof}

\section{Conservation de la finitude du volume}\label{sec_volfini}
\subsection{Les théorèmes de Dirichlet et de Lee}

Le célèbre théorème qui suit nous sera utile pour montrer que le pliage conserve la finitude du volume de la structure projective proprement convexe.

\begin{theo}[Dirichlet]
Soit $\Gamma$ un sous-groupe discret de $\SO$. Il existe un domaine fondamental convexe et localement fini pour l'action de $\G$ sur $\H$.
\end{theo}

Si besoin, on rappelle la définition d'un domaine fondamental.

\begin{defi}
Soient $X$ un espace topologique et $\G$ un groupe qui agit sur
$X$ par homéomorphisme, on dit qu'une partie fermée $D \subset X$ est un \emph{domaine
fondamental pour l'action de $\G$ sur $X$} lorsque:
\begin{itemize}
\item $\underset{\g \in \G}{\bigcup} \g D = X$.

\item $\forall \g \neq 1$, $\g \overset{\circ}{D} \cap
\overset{\circ}{D} = \varnothing$.
\end{itemize}
De plus, un domaine fondamental $D$ pour l'action de $\G$ sur $X$
est dit \emph{localement fini} lorsque:
\begin{itemize}
\item $\forall K$ compact de $X$, $\{\g \in \G \,|\, \g D \cap K
\neq \varnothing \}$ est fini.
\end{itemize}
\end{defi}

\par{
Rappellons aussi très rapidement la démonstration de ce théorème. Pour construire un domaine fondamental, Dirichlet choisi un point $x_0$ dont le stabilisateur dans $\G$ est trivial. Il construit ensuite les hyperplans médiateurs $H_{\g}$ des segments $[x_0,\g \cdot x_0]$, pour $\g \in \G$, ceux sont des hyperplans de $\H$. Ensuite, il montre que l'adhérence de la composante connexe contenant $x_0$ de $\H$ privé de ses hyperplans médiateurs $H_{\g}$ est un domaine fondamental, on l'appele le \emph{domaine de Dirichlet pour l'action de $\G$ sur $\H$ basé en $x_0$}.
}

Le théorème de Dirichlet possède un analogue dans le monde projectif convexe.

\begin{theo}[Jaejeong Lee]\label{lee}
Soient $\O$ un ouvert proprement convexe et $\Gamma$ un sous-groupe discret de $\ss$ qui préserve $\O$. Il existe un domaine fondamental convexe et localement fini pour l'action de $\G$ sur $\O$.
\end{theo}

On pourra trouver une courte démonstration de ce théorème dans \cite{Marquis:2009kq}.

\subsection{Les ellipsoïdes de protection}

\begin{figure}[!h]
\begin{center}
\includegraphics[trim=-4cm 14cm 0cm 0cm, clip=true, width=14cm]{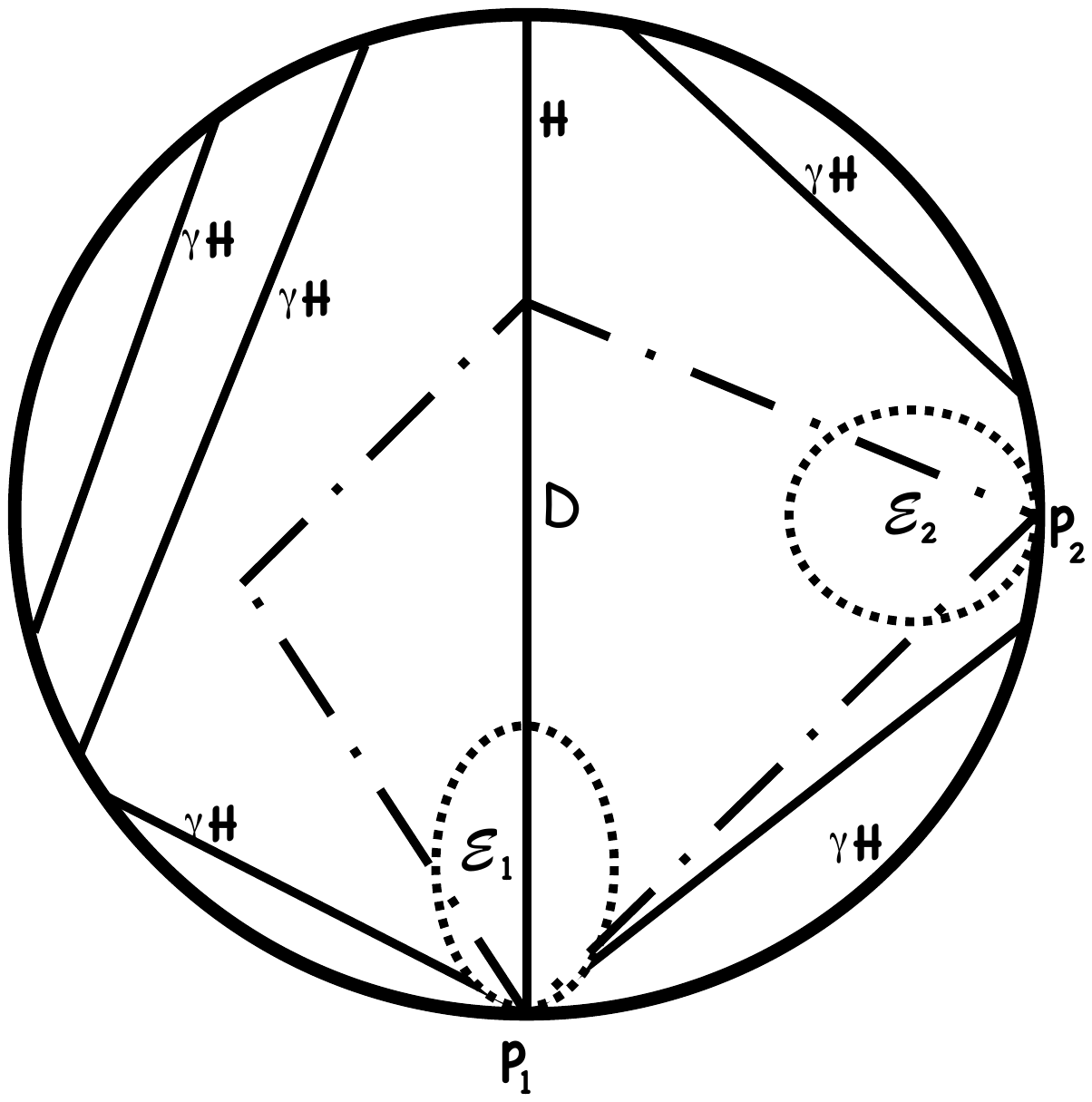}
\caption{Domaine fondamental} \label{domaine}
\end{center}
\end{figure}

\begin{lemm}\label{lem_vol}
Soit $\Gamma$ un sous-groupe discret  sans torsion et de covolume fini de $\SO$. On suppose qu'il existe un hyperplan $H$ de $\H$ tel que  $\Delta = \textrm{Stab}_{\Gamma}(H)$ agisse sur $H$ avec un covolume fini et que l'application $H_{/\Delta} \rightarrow \H_{/\G}$ est un plongement propre.

On se donne $x_{0}$ un point de $H$ dont le stabilisateur dans $\Gamma$ est trivial et on construit alors le domaine de Dirichlet $D$ pour l'action de $\Gamma$ sur $\H$ basé en $x_{0}$.

Le domaine fondamental $D$ rencontre un nombre fini d'hyperplans $\g(H)$ pour $\g \in \G$.

De plus, pour tout point $p \in \partial D$, s'il n'existe pas d'éléments $\g \in \G$ tel que $p \in \overline{\g(H)}$ alors il existe une horoboule centrée en $p$ qui ne rencontre pas les hyperplans $\g(H)$ pour $\g \in \G$.
\end{lemm}

\begin{proof}
\par{
Comme l'application $H_{/\Delta} \rightarrow \H_{/\G}$ est un plongement propre, les hyperplans $(\g(H))_{\g \in \G}$ sont disjoints et forment une famille localement finie dans $\H$. De plus, l'action de $\G$ sur $\H$ est de covolume fini par conséquent, l'ensemble $\partial D$ est fini. Il suffit donc de regarder ce qu'il se passe près des points $p \in \partial D$.
}
\\
\par{
Soit $p \in \partial D$. Comme l'action de $\Gamma$ sur $\H$ est de covolume fini. Le point $p$ est un point parabolique borné, c'est à dire que le groupe $\textrm{Stab}_{\Gamma}(p)$ agit cocompactement sur $\partial \H - \{ p \}$. On note $I = \{ \g \in \G \,\mid \, p \in \overline{\g(H)}  \}$.
}
\\
\par{
 Nous allons nous placer dans le modèle du demi-espace de Poincaré. Supposons que $p = \infty$, une horoboule centrée en l'infini dans le modèle du demi-espace de Poincaré est l'ensemble des points d'altitudes supérieures à une constante. On peut trouver une horoboule centrée en $p$ et qui ne rencontre pas les hyperplans $\g(H)$ pour $\g \notin I$ si et seulement si les altitudes des hyperplans $\gamma(H)$ pour $\gamma \notin I$ sont bornées. L'altitude d'un hyperplan ne passant pas par l'infini dans le modèle du demi-espace de Poincaré est égale au rayon de la sphère qu'il définit sur $\partial \H  - \{ \infty \} = \R^{n-1} \times \{ 0\}$.  Le groupe $\textrm{Stab}_{\Gamma}(p)$ agit par isométrie (euclidienne) et cocompactement sur $\partial \H - \{ \infty \}$, les altitudes sont donc bornées. Il existe donc une horoboule $\E$ centrée en $p = \infty$ tel que si $\g(H) \cap \E \neq \varnothing$ alors  $p \in \overline{\g(H)}$.
}
\\
\par{
En particulier, si le point $p$ ne rencontre pas les hyperplans $\g(H)$ pour $\g \in \G$ alors il existe une horoboule centrée en $p$ qui ne rencontre pas les hyperplans $\g(H)$ pour $\g \in \G$. 
}
\\
\par{
De plus, les hyperplans $\g(H)$ pour $\g \in \G$ ne se rencontrent pas, par conséquent les hyperplans $\g(H)$ pour $\g \in I$ sont parallèles. Le groupe $\textrm{Stab}_{\Gamma}(p)$ préservent ces hyperplans et agit par isométrie (euclidienne) de façon cocompacte sur $\partial \H - \{ p \}$. Le domaine fondamental $D$ est donc inclus dans un cône $\mathfrak{C}$ de sommet $p$ et de base un compact de $\partial \H - \{ p \}$. Comme les $\g(H)$ pour $\g \in I$ sont parallèles et à distance minorée, le cône $\mathfrak{C}$  et donc $D$ ne rencontre qu'un nombre fini d'hyperplans $\g(H)$ pour $\g \in I$.
}


\end{proof}

\begin{defi}
Une \emph{ellisphère} est le bord d'un ellipsoïde.
\end{defi}

\begin{rema}\label{rem_horo}
Soient $\E$ un ellipsoïde de $\PP^n$ et $p$ un point de $\partial \E$. Soit $P'$ le stabilisateur de $p$ dans $\Aut(\E)$. Le groupe $P'$ est isomorphe au groupe des similitudes $\textrm{Sim}^+(\R^{n-1})$. Il possède donc un sous-groupe distingué $P$ isomorphe à $\textrm{Isom}^+(\R^{n-1})$. Il s'agit du sous-groupe des éléments paraboliques qui fixent $p$. Les orbites de $P$ agissant sur $\PP^n$ sont les ellisphères (privé de $p$) du faisceau d'ellisphères engendré par $\partial \E$ et $T_p \E$. Les orbites de $P$ agissant sur $\E$ sont les horosphères de centre $p$ du modèle projectif de l'espace hyperbolique. Il vient donc que toute horosphère d'un ellipsoïde est une ellisphère privé d'un point. La réciproque est par contre fausse. Toute ellisphère incluse avec tangence en un point dans un ellipsoïde n'est pas une horosphère de celle-ci.
\end{rema}

\begin{defi}
Soient $\O$ et $\O'$ deux ouverts proprement convexe de $\PP^n$, on dira que $\O$ est à \emph{l'intérieur} (resp. \emph{l'extérieur}) de $\O'$ lorsque $\O \subset \O'$ (resp. $\O' \subset \O$).
\end{defi}

\begin{lemm}\label{pos_elli}
Soit $\O$ un ouvert proprement convexe et $\G$ un sous-groupe discret de $\Aut(\O)$ qui fixe un point $p \in \partial \O$. Supposons que $\G$ préserve un ellipsoïde $\E$ tangent à $\O$ en $p$ et que l'action de $\G$ sur $\partial \O \setminus \{ p\}$ est cocompact. Alors, il existe un ellipsoïde $\E_{int}$ (resp. $\E_{ext}$) à l'intérieur (resp. l'extérieur) de $\O$. De plus,  $\E_{int}$ est une horoboule de  $\E_{ext}$.
\end{lemm}

\begin{proof}
Soit $A$ une carte affine contenant l'adhérence de l'ouvert proprement convexe $\O$. Soit $D$ un domaine fondamental pour l'action de $\G$ sur $\O$ (Théorème \ref{lee}). L'ensemble $\partial D \cap \partial \O$ est composé de $\{p\}$ et d'un compact de $\partial \O \setminus \{ p\}$. On considère $\mathfrak{C}$ le demi-cône de $A$ de sommet $p$ engendré par $D$.

Les ellisphères du faisceau d'ellisphères engendrés par l'ellisphère $\partial \E$ et l'hyperplan tangent $T_p \partial \O$ sont préservées par $\G$. Le groupe $\Aut(\E)$ est conjugué au groupe $\SO$. Le groupe $\G$ est un sous-groupe de $\Aut(\E)$ composé uniquement d'éléments paraboliques qui fixent $p$.

Par conséquent, pour trouver un ellipsoïde $\E_{int}$ (resp. $\E_{ext}$)  à l'intérieur (resp. l'extérieur) de $\O$. Il suffit de voir que si l'ellisphère $\partial \E'$ du faisceau est suffisament proche  (resp. éloigné) de $p$ alors $\E' \cap \mathfrak{C}$ est inclus dans $D$ (resp. $\partial \E' \cap \mathfrak{C} \cap \O = \varnothing$ ). On peut donc trouver un ellipsoïde $\E_{int}$ (resp. $\E_{ext}$), en prenant un ellipsoïde suffisament proche (resp. éloigné) de $p$.
\end{proof}


\begin{prop}\label{dom_fond_proj}
Il existe un domaine fondamental $D_t$ pour l'action de $\Lambda_t$ sur $\mho_t$ tel que pour tout point $p \in \partial D_t \cap \mho_t$ il existe deux ellipsoïdes $\F_p$ et $\F'_p$ et les points suivants sont vérifiés:
\begin{enumerate}
\item $D_t$ est connexe et c'est une réunion finie de convexe.
\item $\partial D_t$ est fini.
\item $\F_p \subset \mho_t \subset \F'_p$
\item $\partial \F_p \cap \partial\mho_t= \partial \F'_p \cap \partial\mho_t = \{ p \}$
\item $\F_p$ est une horoboule de $\F'_p$.
\end{enumerate}
\end{prop}

\begin{proof}
L'image $D_t$ de $D$ par $\dev_t$ est un domaine fondamental pour l'action de $\Lambda_t$ sur $\mho_t$. Il vérifie les deux premiers points car $D$ ne rencontre qu'un nombre fini d'hyperplans $\g(H)$ pour $\g \in \G$ et l'application $\dev_t$ restreinte à n'importe qu'elle composante connexe de $\displaystyle{\mho_0 \setminus \bigcup_{\g \in \Lambda_0} \g H}$ est une application projective.

Si le point $p \in \partial D_t \cap \partial \mho_t$ et $\displaystyle{p \notin \bigcup_{\g \in \Lambda_t} \g H}$, alors l'ellipsoïde $\E_p$ fourni par le lemme \ref{lem_vol} est inclus dans l'une des composantes connexes de $\displaystyle{\mho_0 \setminus \bigcup_{\g \in \Lambda_0} \g H}$. Mais l'application $\dev_t$ restreinte à n'importe qu'elle composante connexe de $\displaystyle{\mho_0 \setminus \bigcup_{\g \in \Lambda_0} \g H}$ est une application projective. L'ellipsoïde $\F_p = \dev_t(\E_p)$ convient. Comme l'ellipsoïde $\F_p$ est préservé par $\Stab_{\Lambda_t}(p)$ , le lemme \ref{pos_elli} montre qu'il existe un ellipsoïde $\F'_p$ à l'extérieur de $\O$ et tel que $\F_p$ est une horoboule de centre $p$ de $\F'_p$.

Si le point $p \in \partial D_t \cap \partial \mho_t$ et  $\displaystyle{p \in \bigcup_{\g \in \Lambda_t} \g H}$, alors la situation est un peu plus complexe car 
l'ellipsoïde $\E_p$ fourni par le lemme \ref{lem_vol} est inclus dans une infinité de composantes connexes de $\displaystyle{\mho_0 \setminus \bigcup_{\g \in \Lambda_0} \g H}$. L'image de $\E_p$ par $\dev_t$ n'est donc pas un ellipsoïde, c'est ellipsoïde par morceaux. 

Nous allons montrer que le groupe $\Stab_{\Lambda_t}(p)$ préserve un ellipsoïde, ainsi le lemme \ref{pos_elli} montrera qu'il existe deux ellipsoïdes $\F_p$ et $\F'_p$ qui solutionnent notre problème. Montrer que  le groupe $\Stab_{\Lambda_t}(p)$ préserve un ellipsoïde revient à montrer qu'il préserve une forme quadratique de signature $(n,1)$.

Le groupe $\Lambda^t_p = \Stab_{\Lambda_t}(p)$ est virtuellement isomorphe à $\Z^{n-1}$, c'est le groupe fondamental du cusp de $M_n$ associé à $p$. On note $\Lambda^t_{p,H}$ le sous-groupe de $\Lambda^t_p$ des éléments qui préserve $H$. Il est virtuellement isomorphe à $\Z^{n-2}$. Le groupe $\Lambda^t_{p,H}$ n'est pas modifié pendant le pliage, c'est un sous-groupe de $\pi_1(\Nn_0)$.

 On souhaite trouver un ellipsoïde à l'extérieur et à ellipsoïde à l'intérieur à $\mho_t$ et tangent en $p$ à $\partial \mho_t$. Pour cela, il suffit de trouver (lemme \ref{pos_elli}) de trouver un ellipsoïde préservé par un sous-groupe d'indice fini de $\Stab_{\Lambda_t}(p)$. On peut donc supposer que $\Stab_{\Lambda_t}(p)$ est isomorphe à $\Z^{n-1}$.

%

Le groupe $\Lambda^t_p$ préserve le point $p$, il agit donc sur l'espace projectif $\PP^{n-1}_p$ des droites projectives de $\PP^n$ passant par $p$. Il s'agit de l'espace projectif des droites de l'espace vectoriel quotient $\R^{n+1}/p$. Il préserve aussi l'hyperplan $T_p \partial \mho_0$, par conséquent il agit par transformations affines sur l'espace affine $\mathbb{A}^{n-1}_p = \PP^{n-1}_p \setminus (\PP^{n-1}_p \cap T_p \partial \mho_0)$. Cet espace affine est dirigé par l'espace vectoriel quotient $\overline{T_p \partial \mho_0}/p$, où $\overline{T_p \partial \mho_0}$ est le relevé de $T_p \partial \mho_0$ à $\R^{n+1}$.

Montrer que le groupe $\Lambda^t_p$ préserve un ellipsoïde tangent à $\partial \O$ en $p$ revient à montrer que l'action de $\Lambda^t_p$ préserve un produit scalaire sur $\mathbb{A}^{n-1}_p$. Le lemme \ref{lem_eucli} ci-dessous montre qu'il suffit de montrer que le groupe $\Lambda^t_p$ agit par transformation affine de déterminant 1.

Comme les éléments de $\Lambda^t_p$ viennent d'un sous-groupe de $\ss$, il s'agit de montrer que $p$ est un vecteur propre dont la valeur propre associé est 1. Comme $p$ est l'unique point fixe des éléments de $\Lambda^t_p$, le lemme \ref{vp} ci-dessous montre que la valeur propre associé à $p$ est 1.



\end{proof}

\begin{lemm}\label{lem_eucli}
Soit $\G$ un sous-groupe discret agissant proprement et cocompactement sur $\R^m$ par transformation affine de déterminant 1.
Supposons que $\G$ est virtuellement isomorphe à $\Z^m$ et possède un sous-groupe  $\G'$ distingué virtuellement isomorphe à $\Z^{m-1}$ qui agit proprement et cocompactement sur un hyperplan $H$ par transformation euclidienne. Alors, il existe un produit scalaire sur $\R^m$ préservé par $\G$, autrement dit $\G$ agit par transformation euclidienne sur $\R^m$.
\end{lemm}

\begin{proof}
On peut supposer que $\G$ est isomorphe à $\Z^m$,  que $\G'$ est isomorphe à $\Z^{m-1}$ et qu'il existe un élément $\delta \in \G$ tel que $\G=  \G' \oplus <\delta>$. L'adhérence de Zariski du groupe $\G'$ est le groupe des translations préservant l'hyperplan $H$. Un calcul facile montre que $\forall \overrightarrow{u} \in \overrightarrow{H}, \delta t_{\overrightarrow{u}} \delta^{-1} = t_{\overrightarrow{\delta} (\overrightarrow{u})}$. L'élément $\delta$ commute avec les éléments de $\G'$ par conséquent la partie linéaire de $\delta$ sur la direction $\overrightarrow{H}$ est l'identité.

La valeur propre 1 est donc de multiplicité $m-1$ mais l'élément $\delta$ est de déterminant 1 par conséquent 1 est une valeur propre de multiplicité $m$, autrement dit $\delta$ est une translation. Ceci montre que $\G$ préserve un produit scalaire sur $\R^m$.

\end{proof}

\begin{lemm}\label{vp}
Soit $\g \in \ss$. Supposons que $\g$ préserve un ouvert proprement convexe $\O$ de $\PP^n$ alors le rayon spectral $\rho(\g)$ de $\g$ est valeur propre de $\g$. De plus, si $\rho(\g) \neq 1$ alors il existe deux points distincts fixés par $\g$ sur le bord $\partial \O$ de $\O$. 
\end{lemm}

\begin{proof}
Le premier point est le contenu du lemme 3.2 de \cite{MR2195260}. Pour montrer le second point, supposons que $\rho(\g) >1$. On se donne $x \in \O$, la suite $\g^n \cdot x$ converge vers un point $x^+ \in \partial \O$ et le point $x^+$ est une droite propre associé à la valeur propre $\rho(\g)$ pour $\g$. De plus, on a  $\rho(\g^{-1}) <1$, on obtient donc un autre point $x^- \in \partial \O$ par le même procédé. Clairement ces deux points sont distincts.
\end{proof}

\subsection{Conservation de la finitude du volume}

Nous pouvons presque montrer les deux corollaires suivants.

\begin{coro}
Le pliage d'une variété hyperbolique de volume fini le long d'une hypersurface totalement géodésique définit une déformation de structure projective proprement convexe de volume fini.
\end{coro}

\begin{coro}
L'action du groupe $\Lambda_t$ sur l'ouvert proprement convexe $\mho_t$ est de covolume fini.
\end{coro}

Pour montrer ce résultat nous aurons besoin d'un théorème de comparaison des volumes en géométrie de Hilbert. Ce résultat de géométrie de Hilbert est très classique, c'est une conséquence directe de la définition de la distance de Hilbert et de la mesure de Busemann. Pour plus de détails, on pourra consulter \cite{MR2270228}.

\begin{prop}\label{compa}
Soient $\O_1$ et $\O_2$ deux ouverts proprement convexes de $\PP^n$
tels que $\O_1 \subset \O_2$, alors, pour tout borélien $\mathcal{A}$ de $\O_1$, on a
$\mu_{\O_2}(\mathcal{A}) \leqslant \mu_{\O_1}(\mathcal{A})$.
\end{prop}

\begin{proof}[Démonstration des deux corollaires]
Comme $\partial D_t$ est fini (\ref{dom_fond_proj}), il est clair que $D_t$ est de volume fini si et seulement si pour tout $p \in \partial \mho_t \cap D_t$, il existe un voisinage $V_p$ de $p$ dans $\overline{\O}$ tel que $\mu_{\mho_t}(V_p \cap D_t) < \infty$. Mais, le lemme \ref{dom_fond_proj} montre que si $V_p$ est assez petit alors $V_p \cap D_t \subset \F_p$. Le convexe $\F_p$ est un ellipsoïde (donc c'est l'espace hyperbolique), par conséquent, il est bien connu que $\mu_{\F_p}(V_p \cap D_t) < \infty$. Enfin, la proposition \ref{compa} montre que $\mu_{\mho_t}(V_p \cap D_t) \leqslant \mu_{\F_p}(V_p \cap D_t)  < \infty$.
\end{proof}

\subsection{Le groupe $\Lambda_t$ n'est pas un groupe de Schottky}

\begin{rema}
Le groupe $\Lambda_t$ n'est pas un groupe de Schottky. La définition de groupe de Schottky ne fait pas l'unanimité rappelons donc deux définitions pour fixer notre propos.

\begin{defi}
Un élément $\g$ de $\ss$ est dit \emph{loxodromique} lorsque les valeurs propres de $\g^2$ sont simples et positives.
\end{defi}

\begin{defi}
Un sous-groupe $\G$ de $\ss$ est un \emph{groupe de Schottky} lorsque c'est un groupe libre discret dont tous les éléments sont loxodromiques.
\end{defi}

Le groupe $\Lambda_t$ n'est pas un groupe de Schottky car il contient des éléments unipotents (i.e 1 est l'unique valeur propre). Par exemple, les stabilisateurs des points $p \in \partial D_t$.
\end{rema}

\section{Gromov-hyperbolicité}\label{sec_gro}

\subsection{Point de concentration faible}

\begin{defi}
Soient $\O$ un ouvert proprement convexe de $\PP^n$ et $\G$ un sous-groupe discret et irréductible de $\ss$ qui préserve $\O$. Un point $x \in \partial \O$ est \emph{un point parabolique borné} si l'action du groupe $\Stab_{\G}(x)$ sur $L_{\G}- \{ x\}$ est cocompacte.
\end{defi}

\begin{defi}
Soient $\O$ un ouvert proprement convexe de $\PP^n$ et $\G$ un sous-groupe discret de $\ss$ qui préserve $\O$. On dit qu'un point $x \in \partial \O$ est un \emph{point limite conique} lorsqu'il existe une suite d'éléments $(\delta_n)_{n \in \N}$ de $\G$, un point $x_0 \in \O$, une demi-droite $[x_1,x[$, et un réel $C > 0$ tel que:
\begin{enumerate}
\item $\delta_n \cdot x_0 \underset{ n \to \infty}{\to} x$
\item $d_{\O}(\delta_n \cdot x_0, [x_1,x[) \leqslant C$
\end{enumerate}
\end{defi}

\begin{defi}
Soient $\O$ un ouvert proprement convexe de $\PP^n$ et $\G$ un sous-groupe discret de $\ss$ qui préserve $\O$. On dit qu'un point $x \in \partial \O$ est un \emph{point de concentration faible} lorsqu'il existe un voisinage connexe $\U$ de $x$ dans $\partial \O$ tel que pour tout voisinage $\V$ de $x$ dans $\partial \O$, il existe un élément $\g \in \G$ tel que:
\begin{enumerate}
\item $x \in \g(\U)$
\item $\g(\U) \subset \V$
\end{enumerate}
\end{defi}

%

\begin{theo}[Hong, Jeong, SaKong \cite{MR1430961}]\label{theo_core}
Soient $\E$ un ellipsoïde de $\PP^n$ et $\G$ un sous-groupe discret de $\ss$ qui préserve $\E$. Tout point limite conique est un point de concentration faible.
\end{theo}

%
%

Le théorème suivant a de nombreuses versions et de nos nombreux auteurs, on pourra trouver une démonstration dans l'article \cite{MR1218098} de Bowditch.

\begin{theo}\label{theo_point_volfini}
Soient $\E$ un ellipsoïde de $\PP^n$ et $\G$ un sous-groupe discret de $\ss$ qui préserve $\E$. 
L'action de $\G$ sur $\E$ est de covolume fini si et seulement si tout point de $\partial \E$ est un point parabolique borné ou un point limite conique.
\end{theo}

On obtient donc la proposition suivante:

\begin{prop}
Tout point de $\partial \mho_t$ est un point de concentration faible ou un point parabolique borné pour l'action de $\Lambda_t$.
\end{prop}

\begin{proof}
\par{
Le théorème \ref{theo_point_volfini} montre que tout point de $\mho_0$ est un point parabolique borné ou un point limite conique. Le théorème \ref{theo_core} montre tout point de $\mho_0$ est un point parabolique borné ou un point de concentration faible. Ces deux notions ne font intervenir que les propriétés topologiques de l'action de $\Lambda_0$ par homéomorphisme sur $\partial \mho_0$.
}
\\
\par{
L'application $\dev_t$ est un homéomorphisme équivariant entre les deux convexes $\mho_0$ et $\mho_t$ de $\PP^n$. Le théorème \ref{theo_kapo} montre que cet homéomorphisme se prolonge en un homéomorphisme équivariant: $\dev_t: \partial \mho_0 \rightarrow \partial \mho_t$. Par conséquent, tout point de $\mho_t$ est un point parabolique borné ou un point de concentration faible pour l'action de $\Lambda_t$.
}
\end{proof}

\subsection{Strict-convexité}

\begin{prop}\label{non_stric_conv}
Soit $\O$ un ouvert proprement convexe de $\PP^n$, supposons qu'il existe un sous-espace projectif $E$ de $\PP^n$ de dimension $1 \leqslant m \leqslant n-1$ tel que l'intersection $\partial \O \cap E$ soit d'intérieur dans $E$ non vide.
On note $S$ l'intérieur dans $E$ de $\partial \O \cap E$, c'est un ouvert proprement convexe de $E$.

Pour toute suite de points $x_n \in \O$ , pour tout point $x_{\infty} \in \partial \O$, et tout réel $R>0$, si la suite $x_n \underset{n \to \infty}{\to} x_{\infty}$ alors la suite $B^{\O}_{x_n}(R)$ converge vers la boule $B^{S}_{x_{\infty}}(R)$ pour la distance de Hausdorff induite par la distance canonique $d_{can}$ de $\PP^n$.
\end{prop}

\begin{proof}
Il faut montrer deux choses pour montrer cette proposition. Si $X$ est un ensemble, on notera $X^{\varepsilon} = \{ y \in \PP^n \,\mid\, d_{can}(y,X) \leqslant \varepsilon \}$. On doit montrer que pour tout $\varepsilon >0$, il existe $N >0$ tel que pour tout $n \geqslant N$, on a $ B^{S}_{x_{\infty}}(R) \subset B^{\O}_{x_n}(R)^{\varepsilon}$ et $B^{\O}_{x_n}(R) \subset B^{S}_{x_{\infty}}(R)^{\varepsilon}$. Avec des quantificateurs cela se traduit par:

$$\forall \varepsilon>0 ,\, \exists N >0 ,\, \forall n \geqslant N ,\, \forall z_{\infty} \in  B^{S}_{x_{\infty}}(R), \exists y_n \in B^{\O}_{x_n}(R) \textrm{ tel que } d_{can}(y_n,z_{\infty}) \leqslant \varepsilon$$

et

$$
\forall \varepsilon>0 ,\, \exists N >0 ,\, \forall n \geqslant N ,\, \forall y_n \in  B^{\O}_{x_n}(R), \exists z_{\infty} \in B^{S}_{x_{\infty}}(R) \textrm{ tel que } d_{can}(y_n,z_{\infty}) \leqslant \varepsilon
$$

Il s'agit donc de montrer que l'on peut trouver une "tranche" de $B^{\O}_{x_n}(R)$ qui converge "uniformément" vers $B^{S}_{x_{\infty}}(R)$ et que $B^{\O}_{x_n}(R)$ est proche de $B^{S}_{x_{\infty}}(R)$ si $n$ est assez grand.


Pour le montrer, le premier point, on commence par choisir une carte affine $A$ qui contient $\O$.
Considérons, $E_n$ le sous-espace affine parallèle à l'espace affine engendré par $S$ passant par $x_n$ dans la carte $A$. Tout point de l'intersection $C_n = E_n \cap B^{\O}_{x_n}(R)$ converge vers un point de $B^{S}_{x_{\infty}}(R)$, et inversement pour tout point $z_{\infty}$ de $B^{S}_{x_{\infty}}(R)$, il existe une suite de points de $C_n$ qui converge vers $z_{\infty}$.

Le sous-espace affine $E_n$ converge $E$ pour la distance de Haussdorff dans $A$. Par suite, pour tout $\varepsilon >0$, il existe un $N >0$, tel que pour tout $n \geqslant N$, on a $ B^{S}_{x_{\infty}}(R) \subset B^{\O}_{x_n}(R)^{\varepsilon}$. 

Pour le montrer, le second point, il faut remarquer que si $D_n$ est la droite passant par $x_n$ parallèle à une direction fixée  qui n'est pas incluse dans $E$ alors le diamètre pour $d_{can}$ de $D_n \cap B^{\O}_{x_n}(R)$ tend vers $0$ lorsque $n$ tend vers l'infini. Par suite, pour tout $\varepsilon >0$, il existe un $N >0$, tel que pour tout $n \geqslant N$, on ait $  B^{\O}_{x_n}(R) \subset B^{S}_{x_{\infty}}(R)^{\varepsilon}$. 
\end{proof}

\begin{prop}\label{point_conc}
Soient $\O$ un ouvert proprement convexe de $\PP^n$ et $\G$ un sous-groupe discret et irréductible de $\ss$ qui préserve $\O$.  Supposons qu'il existe un sous-espace projectif maximal $E$ de $\PP^n$ de dimension $1 \leqslant m \leqslant n-2$ tel que l'intersection $\partial \O \cap E$ soit d'intérieur dans $E$ non vide. Soit $x$ un point dans l'intérieur relatif de $\partial \O \cap E$, le point $x$ ne peut-être un point de concentation faible.
\end{prop}

\begin{proof}
Supposons que $x$ soit un point de concentration faible. Il existe alors un voisinage $\U$ de $x$ et une suite de d'éléments $(\g_n)_{n \in \N}$ de $\G$ tel que $\g_n(\U) \underset{ n \to \infty}{\to} \{ p \}$. On note $K$, l'enveloppe convexe  de $\U$ dans $\overline{\O}$, les $\g_n$ sont des applications projectives par conséquent, on a $\g_n(K) \underset{ n \to \infty}{\to} \{ p \}$. Comme la dimension de $\partial \O \cap E$ est strictement inférieure à $n-1$ et que $E$ est maximal, le compact $K$ est d'intérieur non vide, il existe donc une boule $B$ de $(\O,d_{\O})$ incluse dans $K$. Toute boule $B$ de $(\O,d_{\O})$ incluse dans $K$ vérifie que $\g_n(B) \underset{ n \to \infty}{\to} \{ p \}$. Mais la proposition \ref{non_stric_conv} montre que la suite $(\g_n(B))_{n \in \N}$ ne peut converger vers le point $p$. Le point $p$ n'est donc pas un point de concentration faible.
\end{proof}

\begin{coro}
L'ouvert proprement convexe $\mho_t$ est strictement convexe.
\end{coro}

\begin{proof}
Tout segment du bord de $\mho_t$ est inclus dans un sous-espace projectif maximal $E$.  L'intersection $R = \partial \mho_t \cap E$ est alors d'intérieur dans $E$ non vide. On note $m$ la dimension de $E$.

On commence par traiter le cas où $m$ vérifie $1 \leqslant m \leqslant n-2$.

Le théorème \ref{theo_point_volfini} montre que tout point de $\mho_t$ est un point parabolique borné ou un point de concentration faible. La proposition \ref{point_conc} montre que tout point $p$ dans l'intérieur relatif de $R$ est un point parabolique borné. Mais, la proposition \ref{dom_fond_proj} montre que si $p$ est un point parabolique borné pour l'action de $\Lambda_t$ sur $\mho_t$ alors il existe un ellipsoïde $\E'_p$ tel que:
\begin{itemize}
\item $\Omega \subset \E'_{p}$.
\item $\partial \E'_{p} \cap \partial \Omega = \{ p \}$.
\end{itemize}
Un tel point $p$ ne peut être dans l'intérieur de $R$.

Enfin, si $m = n-1$ alors les points dans l'intérieur de $R$ ne peuvent être parabolique borné pour la même raison que précédemment. Par conséquent, les points paraboliques bornés ne peuvent être dense dans $\partial \mho_t$. Ce qui est absurde puisque l'action de $\Lambda_t$ sur $\mho_t$ est minimale (proposition \ref{prop_min}).

Par conséquent, l'ouvert proprement convexe $\mho_t$ est strictement convexe.
\end{proof}

\subsection{Gromov-hyperbolicité}

%

\begin{lemm}\label{dem_gro_hyp}
Soit $\Gamma$ un sous-groupe discret de $\ss$ qui pr\'eserve un ouvert proprement convexe $\Omega$ de $\PP^n$.
Soit $D$ un domaine fondamental connexe et localement fini pour l'action de $\Gamma$ sur $\Omega$.
On suppose que:
\begin{enumerate}
\item $D$ est une réunion finie de convexe.
\item $\partial D$ est un ensemble fini.
\item $\forall p \in \partial D \cap \partial \Omega$, il existe deux ellipso\"{i}des ouverts $\E_{p}$ et $\E'_{p}$ tel que:
\begin{itemize}
\item $\E_{p} \subset \Omega \subset \E'_{p}$.
\item $\partial \E_{p} \cap \partial \Omega = \partial \E'_{p} \cap \partial \Omega = \{ p \}$.
\end{itemize}
\item Le convexe $\O$ est strictement convexe.
\end{enumerate}
Alors, l'espace m\'etrique $(\Omega,d_{\Omega})$ est Gromov-Hyperbolique.
\end{lemm}

\begin{proof}
Supposons que l'espace $(\Omega,d_{\Omega})$ n'est pas Gromov-Hyperbolique. Il existe alors une suite $(T_n)_{n \in \N}$ de triangles dont la taille tend vers l'infini lorsque $n \to \infty$.

Il existe donc 4 suites de points de $\O$: $(x_n)_{n \in \N}$, $(y_n)_{n \in \N}$, $(z_n)_{n \in \N}$ et $(u_n)_{n \in \N}$ tels que:
\begin{enumerate}
\item Le triangle $T_n$ a pour sommet les points $x_n,y_n,z_n$.
\item Le point $u_n$ appartient au segment $[x_n,y_n]$.
\item On a $d_{\O}(u_n,[x_n,z_n] \cup [z_n,y_n])\underset{ n \to \infty}{\to} \infty$.
\end{enumerate}

Quitte à appliquer un élément de $\G$, on peut supposer que $u_n \in D$. Quitte à extraire de nouveau, on peut supposer que la suite $(u_n)_{n \in \N}$ converge vers $u_{\infty} \in \overline{D} \subset \overline{\O} \subset \PP^n$. Nous allons distinguer les deux cas suivants.

\begin{enumerate}
\item Le point $u_{\infty} \in \O$.
\item Le point $u_{\infty} \in \partial D \cap \partial \O$.
\end{enumerate}

Commençons par supposer que le point $u_{\infty} \in \O$. Quitte à extraire, on peut supposer que les suites $(x_n)_{n \in \N}$, $(y_n)_{n \in \N}$, $(z_n)_{n \in \N}$ convergent dans $\overline{\O}$ vers les points $x_{\infty},\, y_{\infty} ,\, z_{\infty}$.

Les quantités $d_{\O}(u_n,x_n)$, $d_{\O}(u_n,y_n)$ et $d_{\O}(u_n,z_n)$ tendent vers l'infini lorsque $n \to \infty$. Par conséquent les points $x_{\infty},\, y_{\infty} ,\, z_{\infty}$ sont sur le bord de $\O$.

Il est clair que $x_{\infty}  \neq y_{\infty}$ puisque $u_{\infty} \in \O$. Si $x_{\infty} \neq z_{\infty}$ alors la strict-convexité de $\O$ entraine que $d_{\O}(u_{\infty},[x_{\infty},z_{\infty}] \cup [z_{\infty},y_{\infty}]) < \infty$, ce qui est absurde. Donc $x_{\infty} = z_{\infty}$. Pour la même raison $y_{\infty} = z_{\infty}$. Mais c'est absurde puisque  $x_{\infty}  \neq y_{\infty}$.

Supposons à présent que le point $u_{\infty} \in \partial D \cap \partial \O$. La figure \ref{dem_gro} peut aider à suivre cette partie de la démonstration. Il existe alors deux ellipsoïdes $\E_{p}$ et $\E'_{p}$ tel que:
\begin{enumerate}
\item $\E_{p} \subset \Omega \subset \E'_{p}$.
\item $\partial \E_{p} \cap \partial \Omega = \partial \E'_{p} \cap \partial \Omega = \{ p\}$.
\item $\E'_{p}$ est une horoboule de centre $p$ du convexe $\E_{p}$.
\end{enumerate}

On se donne $p^+$ un point sur le bord de l'ellipsoïde $\E'_{p}$ qui n'est pas $p$ et tel que le segment $[p,p^+]$ rencontre $D$ dans son intérieur. On se donne un élément hyperbolique $g \in \Aut(\E'_{p})$ dont le point attractif est $p^+$ et le point répulsif est $p$. La suite de convexe $g^k \cdot\E_{p}$ tend vers $\E'_{p}$ car $\E_{p}$ est l'horoboule de centre $p$ de l'espace hyperbolique $\E_{p}$. 

Nous allons utiliser $g$ pour obliger la "bisuite" $(g^k \cdot u_n)_{(k,n) \in \N^2}$ à sous-converger dans $\E'_p$. La suite $(u_n)_{n \in \N}$ converge vers $p$ et l'élément $g$ est hyperbolique de point attractif $p^+$. Le supremum $t_{max}$ des quantités $d_{\E'_p}(x, g \cdot x)$ sur la réunion des $g^k(D)$ est fini car le segment $[p,p^+]$ rencontre $D$ dans son intérieur, $D$ est une réunion finie de convexe et $\partial D$ est un ensemble fini.

\begin{figure}[!h]
\begin{center}
\includegraphics[trim=-4cm 14cm 0cm 0cm, clip=true, width=14cm]{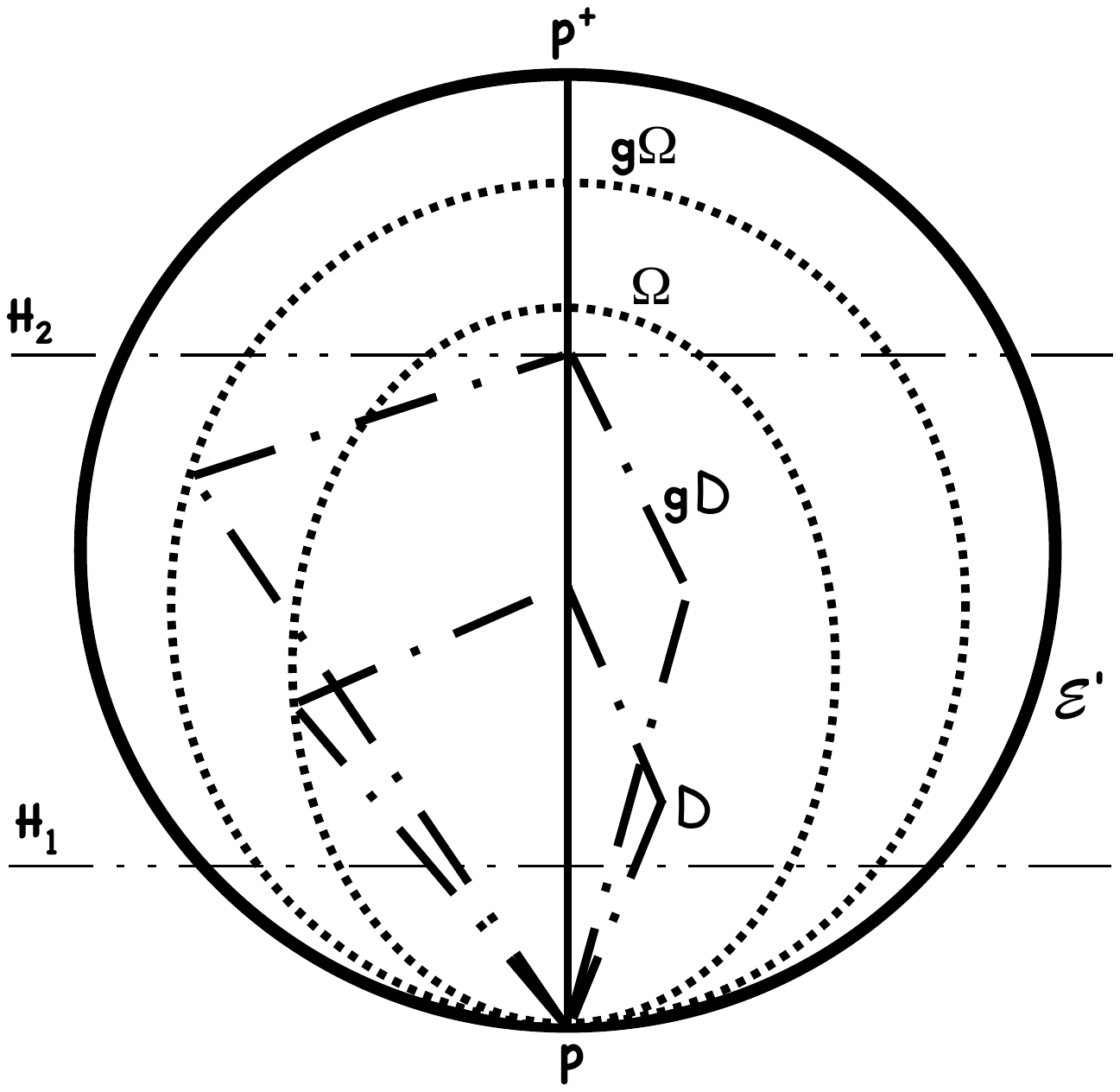}
\caption{Démonstration de la Gromov-Hyperbolicité de $\O$} \label{dem_gro}
\end{center}
\end{figure}

Choisissons un hyperplan $H_1$ qui coupe $\O$ le long d'un ouvert proprement convexe de codimension 1.
L'ouvert $\E'_p \setminus H_1$ est composé de deux composantes connexes, celles "en-dessous" de $H_1$ qui contient $p$ dans son adhérence et celle "au-dessus" de $H_1$ qui contient $p^+$ dans son adhérence. On choisit un second hyperplan $H_2$ parallèle dans $\E'_p$, à distance $h = 10 \times t_{max}$ de $H_1$ et au-dessus de $H_1$.

Après on observe la suite $(u_n)_{n \in \N}$. Il existe un entier $n_0$ tel que pour tout $n \geqslant n_0$, la suite  $u_n$ est sous $H_1$. Ensuite, on applique $g$, la suite $(g^k \cdot u_{n_0})_{k \in \N}$ converge vers $p^+$. La distance entre les hyperplans $H_1$ et $H_2$ est suffisament grande pour qu'il existe un entier $k_0$ tel que $g^{k_0} \cdot u_{n_0}$ soit entre les deux hyperplans.

Ensuite, Il existe un entier $n_1$ tel que pour tout $n \geqslant n_1$, la suite  $g^{k_0} \cdot u_n$ est sous $H_1$. La distance entre les hyperplans $H_1$ et $H_2$ est suffisament grande pour qu'il existe un entier $k_1$ tel que $g^{k_1} \cdot u_{n_1}$ soit entre les deux hyperplans.

Ce procédé montre qu'il existe deux extractrices $(k_i)_{i \in \N}$ et $(n_i)_{i \in \N}$ tels que:
\begin{enumerate}
\item $g^{k_i} \cdot\O \underset{i \to \infty}{\to} \E'_{p}$
\item $g^{k_i} \cdot u_{n_i} \underset{i \to \infty}{\to} u'_{\infty} \in \overline{\E'_{p}}$.
\end{enumerate}

Mais la suite $(u_n)_{n \in \N}$ est incluse dans le domaine fondamental $D$ par conséquent, $u'_{\infty} \in \E'_{p}$.
On peut bien entendu quitte à extraire encore une fois supposer que l'on a aussi les convergences suivantes:

\begin{enumerate}
\item $g^{k_i} \cdot x_{n_i} \underset{i \to \infty}{\to} x'_{\infty} \in \overline{\E'_{p}}$.
\item $g^{k_i} \cdot y_{n_i} \underset{i \to \infty}{\to} y'_{\infty} \in \overline{\E'_{p}}$.
\item $g^{k_i} \cdot z_{n_i} \underset{i \to \infty}{\to} z'_{\infty} \in \overline{\E'_{p}}$.
\end{enumerate}

Les triangles $g^k \cdot T_n$ de $g^k(\O)$ sont isométriques au triangle $T_n$ de $\O$. Par conséquent, on a  $d_{\E'_{p}}(u'_{\infty},[x'_{\infty},z'_{\infty}] \cup [z'_{\infty},y'_{\infty}]) = \infty$. Mais un ellipsoïde muni de sa distance hyperbolique est un espace Gromov-hyperbolique (c'est l'espace hyperbolique réel) donc ceci entraine que $x'_{\infty} = z'_{\infty}$ ou $y'_{\infty} = z'_{\infty}$ et par suite $x'_{\infty} = y'_{\infty}$. Il vient donc que $u'_{\infty} = x'_{\infty} = y'_{\infty}$, ce qui est absurde.
\end{proof}

\begin{coro}
L'ouvert proprement convexe $(\mho_t,d_{\mho_t})$ est Gromov-Hyperbolique.
\end{coro}

\begin{proof}
La proposition \ref{dom_fond_proj} montre que l'action de $\Lambda_t$ sur $\mho_t$ vérifie tous les points du lemme précédent.
\end{proof}

\backmatter
\bibliography{../Biblio/biblio}
\bibliographystyle{alphaurl}

\end{document}